\documentclass[a4paper]{article}
\usepackage[latin1]{inputenc}
\usepackage[T1]{fontenc}
\usepackage[cal=boondox,calscaled=.96]{mathalfa}
\usepackage{amsmath}
\usepackage{amstext}
\usepackage{amscd}
\usepackage{amsfonts}
\usepackage{euscript}
\usepackage{graphicx}
\usepackage{euscript}
\usepackage[debug]{hyperref}
\usepackage{amssymb}
\usepackage{pstricks}

\def\scr{\EuScript}

\long\def\inhibe#1\endinhibe{\relax}

\newcommand{\CC}{\mathbb{C}}

\newcommand{\RR}{R}

\newcommand{\DD}{\mathcal{D}}
\newcommand{\OO}{\mathcal{O}}
\newcommand{\EE}{\mathcal{E}}
\newcommand{\calJ}{\mathcal{J}}

\def\calV{\mathcal{V}}
\def\calK{\mathcal{K}}
\def\calL{\mathcal{L}}
\def\calM{\mathcal{M}}
\def\calQ{\mathcal{Q}}

\newcommand{\Lotimes}{\stackrel{\mathbf{L}}{\otimes}}
\newcommand{\sbullet}{{\hspace{.1em}\scriptstyle\bullet\hspace{.1em}}}

\DeclareMathOperator{\U}{{\bf U}}
\DeclareMathOperator{\Rees}{Rees}
\DeclareMathOperator{\Sim}{Sym}
\DeclareMathOperator{\Der}{Der}
\DeclareMathOperator{\Hom}{Hom}
\DeclareMathOperator{\Ext}{Ext}
\DeclareMathOperator{\SP}{Sp}
\DeclareMathOperator{\fDer}{\mathcal{D}\mathit{er}} 
\DeclareMathOperator{\End}{End}
\DeclareMathOperator{\gr}{\rm gr}
\DeclareMathOperator{\im}{Im}
\DeclareMathOperator{\codim}{codim}
\newcommand{\pcirc}{{\scriptstyle \,\circ\,}}
\newcommand\Id{{\rm Id}}
\DeclareMathOperator{\ann}{\rm ann}
\newcommand{\cd}[3]{D^{#1}_{#2}(#3)}
\DeclareMathOperator{\Dual}{\mathbb{D}}

\newcounter{numero}[section]
\renewcommand{\thenumero}{(\thesection .\arabic{numero})}

\newenvironment{corollary}{\medskip
\refstepcounter{numero}\noindent {\sc  \thenumero\ Corollary.}\
\it}{\vspace{1ex}\par}

\newenvironment{theorem}{\medskip
\refstepcounter{numero}\noindent {\sc  \thenumero\ Theorem.}\
\it}{\vspace{1ex}\par}

\newenvironment{lemma}{\medskip
\refstepcounter{numero}\noindent {\sc  \thenumero\ Lemma.}\
\it}{\vspace{1ex}\par}

\newenvironment{definition}{\medskip
\refstepcounter{numero}\noindent {\sc  \thenumero\ Definition.}\
\it}{\vspace{1ex}\par}

\newenvironment{proposition}{\medskip
\refstepcounter{numero}\noindent {\sc  \thenumero\ Proposition.}\
\it}{\vspace{1ex}\par}

\newenvironment{remark}{\medskip
\refstepcounter{numero}\noindent {\sc  \thenumero\ Remark.}\
}{\vspace{1ex}\par}

\newenvironment{example}{\medskip
\refstepcounter{numero}\noindent {\sc  \thenumero\ Example.}\
}{\vspace{1ex}\par}

\newenvironment{proof}{
\noindent {\sc  Proof.}\ }{\hfill Q.E.D.\vspace{1ex}\par}

\newenvironment{question}{\medskip
\refstepcounter{numero}\noindent {\sc  \thenumero\ Question.}\
}{\vspace{1ex}\par}

\newcommand{\numero}{\refstepcounter{numero}\noindent {\sc  \thenumero\ }}

\title{A duality approach to the symmetry of Bernstein-Sato polynomials of free divisors}

\author{Luis Narv\'{a}ez Macarro\thanks{Partially supported by MTM2010-19298, P12-FQM-2696, MTM2013-46231-P
 and FEDER.}\\
Departamento de \'Algebra \& Instituto de Matem\'aticas (IMUS)\\ University of Sevilla}

\date{June 2015}

\begin{document}

\maketitle

\begin{abstract}
In this paper we prove that the Bernstein-Sato polynomial of any free divisor for which the $\DD[s]$-module $\DD[s] h^s$ admits a Spencer logarithmic resolution satisfies the symmetry property $b(-s-2) = \pm b(s)$. This applies in particular to locally quasi-homogeneous free divisors (for instance, to free hyperplane arrangements), or more generally, to free divisors of linear Jacobian type. We also prove that the Bernstein-Sato polynomial of an integrable logarithmic connection $\EE$ and of its dual $\EE^*$ with respect to a free divisor of linear Jacobian type are related by the equality $b_{\EE}(s)=\pm b_{\EE^*}(-s-2)$. Our results are based on the behaviour of the modules $\DD[s] h^s$ and $\DD[s] \EE[s]h^s $ under duality.
\medskip

\end{abstract}

\section*{Introduction}

In \cite{gran-schul-RIMS-2010} Granger and Schulze proved that the Bernstein-Sato polynomial of any reductive prehomogeneous determinant or of any regular special linear free  divisor satisfies the equality $b(-s-2) = \pm b(s)$. Their proof is based on Sato's fundamental theorem for irreducible reductive prehomogeneous spaces. This symmetry property has been also checked for many other examples of linear (see for instance \cite{gran-schul-RIMS-2010} and \cite{sevenheck_2011}) and non-linear free divisors (e.g. quasi-homogeneous plane curves and the examples in \cite{nakayama_sekiguchi}). In this paper we prove the above symmetry property for free divisors for which the $\DD[s]$-module $\DD[s] h^s$ admits a logarithmic Spencer resolution (see Theorem \ref{teo:main-sym} for a precise statement). This hypothesis holds for any free divisor of linear Jacobian type, and so for any locally quasi-homogeneous free divisor (for instance, free hyperplane arrangements or discriminants of stable maps \cite[Corollary 6.13]{loo_84} in Mather's ``nice dimensions'' \cite{mather_nice_dim}; ``nice dimensions'' are those dimensions of source and target manifolds for which stable proper mappings are dense in the proper mappings).

The main ingredient of the proof is the explicit description of the $\DD[s]$-dual of $\DD[s] h^s$ by means of the logarithmic duality formula in \cite{calde_nar_fourier,calde_nar_ferrara}.
\medskip

Let us mention that for any quasi-homogeneous germ $h:(\CC^d,0) \to (\CC,0)$ with isolated singularity,
 its reduced Bernstein-Sato polynomial $\widetilde{b}(s) = \frac{b(s)}{s+1}$ satisfies the equality
$ \widetilde{b}(s)=\pm \widetilde{b}(-s-d)$.
This result and ours suggest that both are extremal cases of a whole family of ``pure'' cases where symmetry properties occur with other intermediate shiftings (see Question \ref{question:sym-gen}). One can expect even that in the ``non-pure'' cases, the factors of the Bernstein-Sato polynomial which break the symmetry appear as minimal polynomials of the action of $s$ on other $\DD[s]$-modules attached to our singularity (see for instance the examples in 
\cite[\S 3]{narvaez-contemp-2008}), possibly related with the microlocal structure.
\medskip

Let us now comment on the content of the paper.
\medskip

In section \S 1 we recall the different conditions and hypotheses on free divisors we will use throughout the paper.
In section \S 2 we recall the logarithmic Bernstein construction and we study the hypotheses we will need later to prove our main results.
In section \S 3 we apply the duality formula in \cite{calde_nar_ferrara} to describe the $\DD[s]$-dual of $\DD[s] h^{\varphi(s)}$, where $\varphi$ is a $\CC$-algebra automorphism of $\CC[s]$, under the hypotheses studied in section \S 2.
In section \S 4 we prove the symmetry property $b(-s-2) = \pm b(s)$ under the above hypotheses. The idea of the proof is the following: once we know that the 
$\DD[s]$-dual of $\DD[s] h^{s}$ (resp. of $\DD[s] h^{s+1}$) is concentrated in degree $0$ and is isomorphic to $\DD[s] h^{-s-1}$ (resp. to $\DD[s] h^{-s-2}$), we can compute the $\DD[s]$-dual of the exact sequence
$$ 0 \to \DD[s] h^{s+1} \to \DD[s] h^s \to \calQ := \left(\DD[s] h^s\right) / \left(\DD[s] h^{s+1}\right) \to 0$$
and deduce that the $\DD[s]$-dual of $\calQ$ is concentrated in degree $1$ and is isomorphic to $\DD[s] h^{-s-2}/\DD[s] h^{-s-1}$. From here the symmetry property comes up. At the end of the section we give some applications to the logarithmic comparison problem and a characterization of the logarithmic comparison theorem for Koszul free divisors.
In section \S 5 we generalize the above results to the case of integrable logarithmic connections with respect to free divisors of linear Jacobian type. In section \S 6 we have included some open questions dealing with the relationship between the results of \cite{gran-schul-RIMS-2010} and ours, and with the symmetry properties of (reduced) Bernstein-Sato polynomials in the non-free case. Finally, and for the ease of the reader, we have included an Appendix A with a detailed proof of the duality formula \ref{teo:duality-ferrara}, and the needed notions and results about Lie-Rinehart algebras. This material includes a simple proof of the associativity law (see Theorem \ref{teoapp:main-1} and Corollary \ref{coro3:main-1} ), the original proof in \cite{calde_nar_ferrara} being unpleasant.
\medskip

I would like to thank Francisco Castro, Michel Granger and Mathias Schulze for useful discussions and comments. I would also like to thank the referees for their comments. 

\section{Notations and linearity conditions}

In this paper $X$ will denote a complex manifold of pure dimension $d$, $D\subset X$ a hypersurface (= divisor), $\OO_X[\star D]$ the sheaf of meromorphic functions along $D$, $\OO_X(D)$ the sheaf of meromorphic functions along $D$ with poles of order $\leq 1$ and $\DD_X$ the sheaf of linear differential operators with coefficients in $\OO_X$. On $\DD_X[s]$ we will consider two filtrations: the one induced by the usual order filtration $ \DD_X[s]^i = \DD_X^i[s]$, and the {\em total order filtration} $F_T^i \DD_X[s] = \oplus_{p+q=i} \DD_X^p s^q$, $i\geq 0$. The associated graded rings to both filtrations are isomorphic to $(\gr \DD_X)[s]$, but the corresponding gradings on the last sheaf are different.
\medskip

We will also denote by $\calJ_D\subset \OO_X$ the Jacobian ideal of $D\subset X$, i.e. the coherent ideal  of $\OO_X$ whose stalk at any
$p\in X$ is the ideal  generated by $h,h'_{x_1},\dots,h'_{x_d}$,
where $h\in \OO_{X,p}$ is any reduced local equation of $D$ at $p$ and
$x_1,\dots,x_d\in \OO_{X,p}$ is a system of local coordinates centered at $p$.
\medskip

We recall that $D$ is a {\em free divisor}, in the sense of K. Saito \cite{ksaito_log}, if the coherent $\OO_X$-module $\fDer_\CC(-\log D)$ of logarithmic vector fields with respect to $D$ is locally free (of rank $d$). In such a case we will denote by $\calV_X=\OO_X[\fDer_\CC(-\log D)] \subset \DD_X$ the sheaf of logarithmic differential operators with respect to $D$ \cite{calde_ens}. 

\begin{definition} (Cf. \cite[{\S}7.2]{vascon_cmcaag}) Let $A$ be a commutative ring
and $I\subset A$ an ideal. We say that $I$ is of {\em linear type}
if the canonical (surjective) map of graded $A$-algebras $ \Sim_A(I)
\to \Rees(I)$ is an isomorphism.
\end{definition}

In the above definition, if $I=(a_1,\dots,a_r)$ and $(s_{i1},\dots,s_{ir})$, $i\in L$, is a system of generators of the syzygies of $a_1,\dots,a_r$, to say that the ideal $I$ is of linear type is equivalent to saying that any homogeneous polynomial $F(\xi_1,\dots,\xi_r)\in A[\xi]$ such that $F(a_1,\dots,a_r)=0$ is a linear combination with coefficients in $A[\xi]$ of the linear forms $s_{i1}\xi_1+\cdots +s_{ir} \xi_r$, $i\in L$.

\begin{example} \label{ejemplo:regular-sequence} (\cite[Th\'eor\`eme 1]{micali_64}; see also \cite[Proposition 2.4]{calde_nar_compo}) An ideal generated by a regular sequence is of linear type.
\end{example}

\begin{definition} (See \cite[Definitions 1.11, 1.14]{calde_nar_lct_ilc}.)  (a) We say that the divisor $D$ is of {\em linear Jacobian type} at $p\in D$ if $\calJ_{D,p}\subset \OO_{X,p}$ is of linear type. We say that $D$ is of {\em linear Jacobian type} if it is so at any $p\in D$.\medskip

\noindent (b) We say that the divisor $D$ is of {\em differential linear type} at $p\in D$ if for some (and hence for any) reduced local equation $h\in \OO_{X,p}$ of $D$ at $p$, the ideal $\ann_{\DD_{X,p}[s]} h^s$ is generated by order 1 operators (with respect to the usual or to the total order filtration). We say that $D$ is of differential linear type if it is so at any $p\in D$.
\end{definition}

The following Proposition is proven in \cite[Proposition 1.15]{calde_nar_lct_ilc}.

\begin{proposition} \label{prop:LJT->DLT}
Any divisor of linear Jacobian type is of differential linear type.
\end{proposition}

\begin{definition}  (a) We say that the divisor $D$ is {\em strongly Euler homogeneous} if for any $p \in D$ and for some (and hence for any) reduced local equation $h\in \OO_{X,p}$  of $D$ at $p$ there is a germ of vector field $\chi$ at $p$ vanishing at $p$ such that $\chi(h) = h$.
\medskip

\noindent (b) We say that the divisor $D$ is {\em locally quasi-homogeneous} if for any $p \in D$ there is a system of local coordinates $x=(x_1,\dots,x_d)$ centered at $p$ such that the germ $(D,p)$ has a reduced weighted homogeneous defining equation (with strictly positive weights) with respect to $x$.
\end{definition}

The following theorem has been proven in \cite[Theorem 5.6]{calde_nar_compo}.

\begin{theorem} \label{theo:LQH->LJT}
Any locally quasi-homogeneous free divisor is of linear Jacobian type.
\end{theorem}

We do not know any example of a free divisor of linear Jacobian type which is not locally quasi-homogeneous.

\begin{remark}  \label{nota:strong-E-LJT-factor}
There is also the notion of {\em Euler homogeneity}. Namely, we say that the divisor $D$ is {\em Euler homogeneous} at a point $p\in D$ if there is a reduced local equation $h\in \OO_{X,p}$ of $D$ at $p$ and a germ of vector field $\chi$ at $p$ (not necessarily vanishing at $p$) such that $\chi(h) = h$. In that case we also say that $h$ is Euler homogeneous. It is clear that if $D$ is Euler homogeneous at $p$, then it is also Euler homogeneous at any point $q\in D$ close enough to $p$. Notice that not any local reduced equation of an Euler homogeneous divisor is Euler homogeneous. Notice also that for any divisor $D\subset X$, which might not be Euler homogeneous, the divisor $D'=D\times 	\CC \subset X'=X\times \CC$ is always Euler homogeneous. Nevertheless, a divisor $D\subset X$ is strongly Euler homogeneous if and only if $D'=D\times \CC \subset X'=X\times \CC$ is strongly Euler homogeneous. Let us also notice that a divisor $D\subset X$ is of linear Jacobian type at $p\in D$ if and only $D'=D\times \CC \subset X'=X\times \CC$ is of linear Jacobian type at $(p,0)\in D'$.
\end{remark}

\begin{example} Any smooth hypersurface is of linear Jacobian type. More generally, any quasi-homogeneous (with strictly positive weights) isolated singularity is of linear Jacobian type, since any reduced equation $h$ belongs to the ideal generated by the partial derivatives and so the Jacobian ideal is generated by the regular sequence $h'_{x_1},\dots,h'_{x_d}$ (see Example \ref{ejemplo:regular-sequence}).
\end{example}

\begin{proposition} \label{prop:LJT->SEH} If the divisor $D$ is of linear Jacobian type, then it is strongly Euler homogeneous.
\end{proposition}

\begin{proof}  Let $h\in \OO_{X,p}$ be a reduced local equation of $(D,p)$ and $x=(x_1,\dots,x_d)$ a system of local coordinates centered at $p$. We recall the argument in \cite[Remark 1.26 (a)]{calde_nar_lct_ilc} for the ease of the reader. Since $h$ belongs to the integral closure of the ideal $I=(h'_{x_1},\dots,h'_{x_d})$ (cf. \cite[\S 0.5, 1]{MR0374482}), there is a homogeneous polynomial $F\in
\OO_{X,p}[s,\xi_1,\dots,\xi_d]$ of degree $m>0$ such that
$F(h,h'_{x_1},\dots,h'_{x_d})=0$ and $F(s,0,\dots,0)=s^m$. Let $$
\delta_i = \sum_{j=1}^d a_{ij} \frac{\partial }{\partial x_j},\quad
1\leq i\leq n$$ be a system of generators of $\Der(-\log D)_p$ and let
us write $\delta_i(h)=\alpha_i h$. In other words,
$(-\alpha_i,a_{i1},\dots,a_{id})$, $1\leq i\leq n$, is a system of
generators of the syzygies of $h,h'_{x_1},\dots,h'_{x_d}$. From our hypothesis, the polynomial $F$ must be a
linear combination of the polynomials
$$ -\alpha_i s + a_{i1} \xi_1 + \cdots +a_{id}\xi_d, \quad 1\leq i\leq m$$
with coefficients in $\OO_{X,p}[s,\xi]$. Putting $\xi_1=\cdots=\xi_d=0$ we deduce that at least one of the
$\alpha_i$ must be a unit, i.e. $h \in (h'_{x_1},\dots,h'_{x_d})$ and $h$ is Euler homogeneous.
\smallskip

Once we know that $h$ is Euler homogeneous, let us prove that $h$ is strongly Euler homogeneous by induction on the ambient dimension $d$. The case $d=1$ is obvious.
For $d>1$, if the Euler vector field $\chi$ ($\chi(h)=h$) does not vanish at $p$, we can integrate it and prove that $(D,X,p) \simeq (D' \times \CC,\CC^{d-1}\times \CC,(0,0))$, where $(D',0)\subset (\CC^{d-1},0)$ is a germ of a divisor. First, we deduce that $D'$ is of linear Jacobian type, and second, from the induction hypothesis, that $D'$ is strongly Euler homogeneous and so $D$ is strongly Euler homogeneous too (see Remark \ref{nota:strong-E-LJT-factor}).
\end{proof}

Let us recall that a free divisor $D$ is said to be {\it Koszul} (\cite[Definition 4.1.1]{calde_ens}) at $p\in D$ if for some (and hence any) basis $\delta_1,\dots,\delta_d$ of $\fDer_\CC(-\log D)_p$, the sequence $\sigma(\delta_1), \dots, \sigma(\delta_d)$ is regular in $\gr \DD_{X,p}$. It turns out that this property is equivalent to being holonomic in the sense of Saito \cite[Theorem 7.4]{gran-mon-nieto-schul}.
\medskip

The following definition is inspired by \cite[Definition 7.1]{gran-schul-RIMS-2010}, which only applies to the case of linear free divisors (see also Proposition 7.2 and the subsequent remark in \cite{gran-schul-RIMS-2010}.

\begin{definition} Assume that $D$ is a free divisor. We say that $D$ is {\em strongly Koszul} at $p\in D$ if for some (and hence any) basis $\delta_1,\dots,\delta_d$ of $\fDer_\CC(-\log D)_p$ and for some (and hence any) reduced equation $h\in\OO_{X,p}$ of $(D,p)$, the sequence 
$$h,\sigma(\delta_1)- \alpha_1 s, \dots, \sigma(\delta_d)- \alpha_d s,\quad \text{with }\  \delta_i(h)=\alpha_i h,$$ is regular in $\gr \DD_{X,p}[s]$ (since the sequence is formed by homogeneous elements, its regularity does not depend on the order).
\end{definition}

\begin{proposition}  \label{prop:LJT-SK}
Assume that $D$ is a free divisor and $p\in D$. The following properties are equivalent:
\begin{enumerate}
\item[(a)] $D$ is of linear Jacobian type at $p$.
\item[(b)] $D$ is strongly Koszul at $p$.
\end{enumerate}
\end{proposition}

\begin{proof}
Let $x_1,\dots,x_d\in \OO:=\OO_{X,p}$ be a system of local
coordinates centered at $p$, $h\in\OO$ a reduced
local equation of $D$ at $p$ and $J=\calJ_{D,p}=(h,h'_{x_1},\dots,h'_{x_d})$. 
Let $\{\delta_i = \sum_{j=1}^d a_{ij} \frac{\partial}{\partial
x_j}\}_{1\leq i \leq d}$ be a basis of $\fDer(-\log D)_p$,  
and let us write
$\delta_i(h)=\alpha_i h$ and $\sigma_i := \sigma(\delta_i)= \sum_{j=1}^d a_{ij}\xi_j \in \gr \DD_{X,p} =\OO[\xi]$.
The family
$\{(-\alpha_i,a_{i1},\dots,a_{id})\}_{1\leq i \leq d}$ is a basis of
the syzygies of $h,h'_{x_1},\dots,h'_{x_d}$.
\medskip

\noindent (a) $\Rightarrow$ (b):  From Proposition \ref{prop:LJT->SEH} we know that $h$ is Euler homogeneous, i.e. $h\in (h'_{x_1},\dots,h'_{x_d})$, and we can
take $\alpha_1=\cdots=\alpha_{d-1}=0$ and $\alpha_d=1$. In other
words, $\{(a_{i1},\dots,a_{id})\}_{1\leq i \leq d-1})$ is a basis of
the syzygies of $h'_{x_1},\dots,h'_{x_d}$.
\medskip

Let $\varphi : \OO[\xi] \xrightarrow{}
\Rees(J)=\OO[h'_{x_1}t,\dots,h'_{x_d}t]$ be the surjective map of
$\OO$-algebras defined by $\varphi(\xi_i)=h'_{x_i}t$. Since $J$ is
an ideal of linear type, the kernel of $\varphi$ is generated by the
$\sigma_i$, $1\leq
i\leq d-1$. So
$$\dim \left(
\frac{\OO[\xi]}{(\sigma_1,\dots,\sigma_{d-1})}\right)
=\dim \Rees(J) = d+1$$ and $\sigma_1,\dots,\sigma_{d-1}$ is a
regular sequence in $\OO[\xi]$.
\medskip

On the other hand, since $\ker \varphi =
(\sigma_1,\dots,\sigma_{d-1})$ is a prime ideal and $h\notin \ker
\varphi$, we deduce that $h,\sigma_1,\dots,\sigma_{d-1}$ is also
a regular sequence in $\OO[\xi]$ and, indeed $h,\sigma_1,\dots,\sigma_{d-1},\sigma_{d}-s$
is a regular sequence in $\OO[\xi,s]= \gr \DD_{X,p}[s]$ and 
$D$ is strongly Koszul at $p$.
\medskip

\noindent (b) $\Rightarrow$ (a): Assume that $X$ is a small enough open neighborhood of $p$, $\calK^{(1)}=(\sigma_1-\alpha_1 s,\dots, \sigma_d-\alpha_d s)\subset \OO_X[s,\xi_1,\dots,\xi_d]$
 and let $\calK$ be the kernel of the canonical graded surjective map
\begin{equation} \label{eq:Phi}
\Phi: \OO_X[s,\xi_1,\dots,\xi_d] \xrightarrow{} \Rees (\calJ_D),\quad s \mapsto ht,\ \xi_i \mapsto h'_{x_i}t.
\end{equation}
The homogeneous components of $\calK^{(1)}$ and $\calK$ are coherent $\OO_X$-modules. Since $D$ is of linear Jacobian type at any smooth point, we deduce that $\calK/\calK^{(1)}$ is supported by the singular locus of $D$. In particular, for any homogeneous polynomial
$F\in \calK_p$ there is an $N>0$ such that
$h^N F \in \calK^{(1)}_p$, but $h,\sigma_1- \alpha_1 s,\dots, \sigma_d- \alpha_d s$
is a regular sequence and so
$F\in \calK^{(1)}_p$.  We deduce that $\calK_p=\calK^{(1)}_p$ and $D$ is of linear Jacobian type at $p$.
\end{proof}

\begin{corollary} 
Assume that $D$ is a free divisor and $p\in D$. The following properties are equivalent:
\begin{enumerate}
\item[(a)] $D$ is strongly Koszul at $p$.
\item[(b)] $D$ is Euler homogeneous at $p$ and for any reduced Euler homogeneous equation $h\in\OO_{X,p}$ of $(D,p)$ and any basis $\delta_1,\dots,\delta_d$ of $\fDer_\CC(-\log D)_p$ with $\delta_1(h)=\cdots=\delta_{d-1}(h)=0$ and $\delta_d(h)=h$, the sequence $h,\sigma(\delta_1), \dots, \sigma(\delta_{d-1})$ is regular in $\gr \DD_{X,p}$.
\item[(c)] There is a reduced equation $h\in\OO_{X,p}$ of $(D,p)$ and a basis $\delta_1,\dots,\delta_d$ of $\fDer_\CC(-\log D)_p$ with $\delta_1(h)=\cdots=\delta_{d-1}(h)=0$ and $\delta_d(h)=h$ such that the sequence $h,\sigma(\delta_1), \dots, \sigma(\delta_{d-1})$ is regular in $\gr \DD_{X,p}$.
\end{enumerate}
\end{corollary}

\begin{proof} It is a straightforward consequence of Propositions \ref{prop:LJT-SK} and \ref{prop:LJT->SEH}.
\end{proof}

Let us notice that property (c) in the above corollary appeared as condition (c') in \cite[Corollary 1.8]{torrelli_2004}.
\medskip

The following notion was introduced in \cite[page 257]{narvaez-contemp-2008} and was called ``(GK)''.

\begin{definition} Assume that $D$ is a free divisor. We say that $D$ is {\em weakly Koszul} at $p\in D$ if for some (and hence any) basis $\delta_1,\dots,\delta_d$ of $\fDer_\CC(-\log D)_p$ and some (and hence any) reduced local equation $h\in \OO_{X,p}$ of $(D,p)$, the sequence 
$$\sigma(\delta_1)- \alpha_1 s, \dots, \sigma(\delta_d)- \alpha_d s,\quad \text{with }\  \delta_i(h)=\alpha_i h,$$ is regular in $\gr \DD_{X,p}[s]$. We say that $D$ is weakly Koszul if it is so at any $p\in D$.
\end{definition}

\begin{proposition} \label{prop:SK->K->WK}
For a free divisor, the following implications hold:
\begin{enumerate}
\item[(a)] strongly Koszul $\Rightarrow$ Koszul.
\item[(b)] Koszul $\Rightarrow$ weakly Koszul.
\end{enumerate}
\end{proposition}

\begin{proof} The first implication is a consequence of Proposition \ref{prop:LJT-SK} and
\cite[Proposition 1.27]{calde_nar_lct_ilc}. The second one comes from \cite[Proposition 2.2.14]{narvaez-contemp-2008},  \cite[Proposition 1.22]{calde_nar_lct_ilc}.
\end{proof}

The example $x_1 x_2 (x_1+x_2) (x_1+x_3 x_2)=0$ is a weakly Koszul free divisor which is not Koszul \cite[Example 3.1]{narvaez-contemp-2008} and any non-quasihomogeneous plane curve is a Koszul free divisor non-strongly Koszul (Proposition 2.3.1 in \cite{narvaez-contemp-2008}).

\begin{remark} \label{nota:GK-Euler}
Let $D$ be a free divisor and $h\in \OO_{X,p}$ a reduced local equation of $(D,p)$. If there is a germ of vector field $\chi$ at $p$ such that $\chi(h)=h$ (i.e. $h$ is Euler homogeneous), then $D$ is weakly Koszul at $p$ if and only if for some (and hence any) basis $\delta_1,\dots,\delta_{d-1}$ of
germs of vector fields vanishing on $h$, the sequence 
$\sigma(\delta_1), \dots, \sigma(\delta_{d-1})$ is regular in $\gr \DD_{X,p}$.
\end{remark}

\section{Logarithmic--meromorphic comparison for Bernstein modules}  \label{sec:log-mero}

From now on we assume that $h:(\CC^d,0) \to (\CC,0)$ is a reduced local equation of a germ of a free divisor $(D,0) \subset (\CC^d,0) $. Let us write for short $\OO= \OO_{\CC^d,0}$, $\DD = \DD_{\CC^d,0}$ and $\calV=\calV_{\CC^d,0}=\OO[\fDer_\CC(-\log D)_0] \subset \DD$. We consider the {\em logarithmic Bernstein module} $\OO[s] h^s$ \cite[\S 1.6]{calde_nar_lct_ilc}, which is a $\calV[s]$-submodule of the Bernstein $\DD[s]$-module $\OO[s,h^{-1}] h^s$ \cite{bernstein-1972}. Obviously $\OO[s] h^s$ is generated by $h^s$ over $\calV[s]$ and $\ann_{\calV[s]} h^s$ is the left $\calV[s]$-ideal generated by the Lie-Rinehart algebra over $(\CC,\OO)$ (see Appendix A)
$$ \Theta_{h} := \left\{\delta - \alpha s\ |\ \delta \in \fDer(-\log D)_0, \delta(h)= \alpha h \right\} \subset \calV[s].$$

The following result generalizes \cite[Proposition 4.4]{torrelli_2004} to the non-Euler homogeneous case and completes Proposition \ref{prop:LJT-SK}.

\begin{proposition} \label{prop:improve-torrelli} With the above hypotheses, the following properties are equivalent:
\begin{enumerate}
\item[(a)] $(D,0)$ is of differential linear type and weakly Koszul.
\item[(b)] $(D,0)$ is strongly Koszul (or equivalently, of linear Jacobian type).
\end{enumerate}
\end{proposition}

\begin{proof}  (b) $\Rightarrow$ (a):  It is a consequence of Propositions \ref{prop:LJT-SK}, \ref{prop:LJT->DLT} and \ref{prop:SK->K->WK}.
\medskip

\noindent (a) $\Rightarrow$ (b): We follow Torrelli's argument in 3 $\Rightarrow$ 4 of \cite[Proposition 4.4]{torrelli_2004}. Let $\delta_1,\dots,\delta_d$ be a basis of $\fDer_\CC(-\log D)_0$ with $\delta_i(h)=\alpha_i h$ and let us write $K=\ann_{\DD[s]} h^s$. It is clear that $ \Theta_{h}$ is freely generated as $\OO$-module by $\delta_1-\alpha_1 s,\dots,\delta_d -\alpha_d s$. Since $(D,0)$ is of differential linear type, we have $K=\DD[s] \Theta_{h}$. 
Since $\sigma(\delta_1)-\alpha_1 s,\dots,\sigma(\delta_d) -\alpha_d s$ is a regular sequence in $\gr \DD[s]= \gr_{F_T}(\DD[s])$, we deduce that 
$\sigma_{F_T}(K)$ is the ideal of $\gr_{F_T}(\DD[s])$ generated by $\sigma(\delta_1)-\alpha_1 s,\dots,\sigma(\delta_d) -\alpha_d s$. We know that
the characteristic variety $\widetilde{W}=V(\sigma_{F_T}(K))\subset \CC \times T^* \CC^d$ of $\DD[s] h^s$ is irreducible of dimension $d+1$ (\cite[\S 5]{kashiwara-1976}, \cite[Proposition 2.3]{yano_1978}).
In fact $I(\widetilde{W}) = \ker \Phi$, where $\Phi$ has been defined in (\ref{eq:Phi}). Since $\Phi(h)\neq 0$ we deduce that $\dim V(h,\sigma(\delta_1)-\alpha_1 s,\dots,\sigma(\delta_d) -\alpha_d s) = \dim (W \cap V(h)) = d$ and so $h,\sigma(\delta_1)-\alpha_1 s,\dots,\sigma(\delta_d) -\alpha_d s$ is a regular sequence.
\end{proof}
\medskip

Let us denote by $\SP^\sbullet_{\Theta_{h},\calV[s]}= \calV[s] \otimes_{\U(\Theta_{h})} \SP^\sbullet_{\Theta_{h}}$, where the complex $\SP^\sbullet_{\Theta_{h}}$ is defined in  \ref{nume:SP}. From Proposition 1.21 in \cite{calde_nar_lct_ilc}, we know that $\SP^\sbullet_{\Theta_{h},\calV[s]}$ becomes a 
 $\calV[s]$-free resolution of $\OO[s] h^s$ with the augmentation
$ \varepsilon^0:\SP^{0}_{\Theta_{h},\calV[s]} = \calV[s]\to  \OO[s] h^s
$, $\varepsilon^0(P)=P h^s$.
\medskip

The proof of the following proposition is clear.

\begin{proposition} \label{prop:gen-LJT} Under the above hypotheses,
the following properties are equivalent:
\begin{enumerate}
\item[(a)] The canonical map
$ \DD[s] \Lotimes_{\calV[s]} \left( \OO[s]h^s \right) \xrightarrow{} \DD[s] h^s$ is an isomorphism in the derived category of left $\DD[s]$-modules.
\item[(b)] The divisor $D$ is of differential linear type at $0$ and the complex $\DD[s]\otimes_{\calV[s]} \SP^\sbullet_{\Theta_{h},\calV[s]}$ is exact in degrees $\neq 0$.
\end{enumerate}
\end{proposition}

\begin{proposition} Any germ of free divisor $(D,0)\subset (\CC^d,0)$ of differential linear type and weakly Koszul at $0$ satisfies the equivalent properties of Proposition \ref{prop:gen-LJT}.
\end{proposition}

\begin{proof} To prove that the complex $\DD[s]\otimes_{\calV[s]} \SP^\sbullet_{\Theta_{h},\calV[s]}$ is exact in degrees $\neq 0$, we filter it in such a way that 
its graded complex is the Koszul complex associated with the sequence $\sigma(\delta_1)- \alpha_1 s, \dots, \sigma(\delta_d)- \alpha_d s$ with $\delta_1,\dots,\delta_d$ a basis of $\fDer_\CC(-\log D)_0$ and $\delta_i(h)=\alpha_i h$ (see \cite[Proposition 1.18]{calde_nar_lct_ilc}).
\end{proof}

The following corollary is a particular case of \cite[Theorem 3.1]{calde_nar_lct_ilc}.

\begin{corollary}  \label{cor:LJT-properties}
Under the above hypotheses, if $(D,0)\subset (\CC^d,0)$ is a germ of a free divisor of linear Jacobian type, then the equivalent properties of Proposition \ref{prop:gen-LJT} hold.
\end{corollary}

\begin{proof} It is clear from Proposition \ref{prop:improve-torrelli}.
\end{proof}

\begin{definition} \label{def:q-bernstein} For any polynomial $q(s)\in\CC[s]$ we define:
\smallskip

\noindent (1) 
 The {\it $q(s)$-Bernstein module} as the free $\OO[s,h^{-1}]$-module $\OO[s,h^{-1}] h^{q(s)}$ with basis $h^{q(s)}$
endowed with the left $\DD[s]$-module structure given by
$$ \delta \cdot (a h^{q(s)}) = \left( \delta(a) + q(s) \delta(h) h^{-1} a \right) h^{q(s)}$$
for any $\delta \in \Der_\CC(\OO)$.
\smallskip

\noindent (2)  The {\it logarithmic $q(s)$-Bernstein module} as the left $\calV[s]$-submodule $\OO[s] h^{q(s)}$ of the $q(s)$-Bernstein module $\OO[s,h^{-1}] h^{q(s)}$.
\end{definition}

It is clear that $\OO[s] h^{q(s)}$ is generated by $h^{q(s)}$ over $\calV[s]$ and 
$\ann_{\calV[s]} h^{q(s)} $ is the left $\calV[s]$-ideal generated by the $(\CC,\OO)$-Lie-Rinehart algebra
$$ \Theta_{h,q(s)} := \left\{\delta - \alpha q(s)\ |\ \delta \in \fDer(-\log D)_0, \delta(h)= \alpha h \right\} .$$

For any $\CC$-algebra map $\varphi:\CC[s] \to \CC[s]$ let us also call $\varphi$ its trivial extensions to $\OO[s]$, $\OO[s,h^{-1}]$, $\DD[s]$ and $\calV[s]$. For any $q(s)\in\CC[s]$ the map
$$ \overline{\varphi}: a h^{q(s)} \in \OO[s,h^{-1}] h^{q(s)} \mapsto \varphi(a) h^{\varphi(q(s))} \in \OO[s,h^{-1}] h^{\varphi(q(s))}
$$
$$ \text{(resp.\ }\ \overline{\varphi}: a h^{q(s)} \in \OO[s] h^{q(s)} \mapsto \varphi(a) h^{\varphi(q(s))} \in \OO[s] h^{\varphi(q(s))} \text{)}$$
is linear over $\varphi:\DD[s] \to \DD[s]$ (resp. over $\varphi:\calV[s]\to \calV[s]$):
$$ \overline{\varphi} \left(P(s) h^{q(s)}\right) = \varphi(P(s)) h^{\varphi(q(s))},\quad P(s) \in \DD[s].$$
In particular, $\overline{\varphi} \left( \DD[s]h^{q(s)} \right) \subset \DD[s] h^{\varphi(q(s))}$.

For any left $\DD[s]$-module $M$, let us call $\varphi^*(M) := \DD[s] \otimes_{\varphi} M$  the scalar extension associated with $\varphi : \DD[s] \to \DD[s]$. In $\varphi^*(M)$ one has $(P \varphi(Q)) \otimes m = P \otimes (Q m)$ for $m\in M$ and  $P, Q\in \DD[s]$. In a similar way we define $\varphi^*(M) := \calV[s] \otimes_{\varphi} M$ for any left $\calV[s]$-module $M$.
\medskip

Since $\varphi\left( \Theta_{h,q(s)} \right) = \Theta_{h,\varphi(q(s))}$, the map 
$$ \widetilde{\varphi}: \varphi^* \left( \OO[s] h^{q(s)} \right) = \calV[s] \otimes_{\varphi}  \left( \OO[s] h^{q(s)} \right)  \xrightarrow{} \OO[s] h^{\varphi(q(s))}
$$ 
induced by $\overline{\varphi}: \OO[s] h^{q(s)} \to \OO[s] h^{\varphi(q(s))} $ is an isomorphism of left $\calV[s]$-modules:
\begin{eqnarray*}  & \calV[s] \otimes_{\varphi}  \left( \OO[s] h^{q(s)} \right)  \simeq \calV[s] \otimes_{\varphi}  \left( \calV[s]/ \calV[s] \cdot \Theta_{h,q(s)}  \right) \simeq  & \\ & \calV[s]/ \calV[s] \cdot \varphi\left( \Theta_{h,q(s)} \right)   \simeq \calV[s]/ \calV[s] \cdot  \Theta_{h,\varphi(q(s))} \simeq \OO[s] h^{\varphi(q(s))}.&
\end{eqnarray*}

It is clear that 
 $\varphi \left( \ann_{\DD[s]} h^{q(s)} \right) \subset \ann_{\DD[s]} h^{\varphi(q(s))}$. If moreover $\varphi$ is an automorphism, this inclusion becomes an equality. This shows the following lemma.

\begin{lemma} \label{lema:varphi} If $\varphi:\CC[s] \to \CC[s]$ is an automorphism, then the map
$$ \widetilde{\varphi}:\varphi^* \left( \DD[s] h^{q(s)} \right) := \DD[s] \otimes_{\varphi} \left( \DD[s] h^{q(s)}  \right) \xrightarrow{} \DD[s] h^{\varphi(q(s))}$$ 
induced by $\overline{\varphi}: \DD[s]h^{q(s)} \to \DD[s] h^{\varphi(q(s))}$ is an isomorphism of left $\DD[s]$-modules.
\end{lemma}

\begin{proposition}  \label{prop:gen-LJT-varphi}
Assume that $\varphi:\CC[s] \to \CC[s]$ is an automorphism of $\CC$-algebras. Then, the following properties are equivalent to the properties of Proposition \ref{prop:gen-LJT}:
\begin{enumerate}
\item[(a')] The canonical map
$ \DD[s] \Lotimes_{\calV[s]} \left( \OO[s]h^{\varphi(s)}  \right) \xrightarrow{} \DD[s] h^{\varphi(s)} $ is an isomorphism in the derived category of left $\DD[s]$-modules.
\item[(b')] $\ann_{\DD[s]} h^{\varphi(s)} $ is the left $\DD[s]$-ideal generated by $\Theta_{h,\varphi(s)} $ and the  complex $\DD[s]\otimes_{\calV[s]} \SP^\sbullet_{\Theta_{h,\varphi(s)},\calV[s]}$ is exact in degrees $\neq 0$, where $\SP^\sbullet_{\Theta_{h,\varphi(s)},\calV[s]}$ is defined in a completely similar way to $\SP^\sbullet_{\Theta_{h},\calV[s]}$.
\end{enumerate}
\end{proposition} 

\begin{proof} Since $\varphi$ is an automorphism, the functors $\varphi^*$ are exact. On the other hand we obviously have
$\varphi^* \left( \DD[s] \otimes_{\calV[s]} - \right) \simeq \DD[s] \otimes_{\calV[s]} \varphi^*(-)$ and so
$$ \varphi^* \left( \DD[s] \Lotimes_{\calV[s]} \OO[s]h^s \right) \simeq \DD[s] \Lotimes_{\calV[s]} \varphi^*(\OO[s]h^s) \simeq 
\DD[s] \Lotimes_{\calV[s]} \OO[s]h^{\varphi(s)}.$$
The equivalence between (a') and property (a) in Proposition \ref{prop:gen-LJT} comes from Lemma \ref{lema:varphi}.
The equivalence between (a') and (b') comes from the fact that $\SP^\sbullet_{\Theta_{h,\varphi(s)},\calV[s]}$ is a free resolution of the left $\calV[s]$-module $\OO[s]h^{\varphi(s)}$. Let us also notice that $\SP^\sbullet_{\Theta_{h,\varphi(s)},\calV[s]} \simeq \varphi^*\left( \SP^\sbullet_{\Theta_{h},\calV[s]}\right)$.
\end{proof}

\section{Duality} \label{sec:dual}

In this section we first state the duality formula associated with the ring extension $\calV[s] \subset \DD[s]$, which is a particular case of the duality formula in 
\cite[th. (4.5)]{calde_nar_ferrara}. We refer to Appendix A for the details. As an application, we compute the $\DD[s]$-dual of $\DD[s] h^{\varphi(s)}$, for $\varphi$ and automorphism of the $\CC$-algebra $\CC[s]$, under the assumption that the equivalent properties of Proposition \ref{prop:gen-LJT} hold.
\medskip

We keep the notations of \S \ref{sec:log-mero}. The free $\OO$-modules $L=\fDer_\CC(-\log D)_0$ and $L'=\fDer_\CC(\OO_X)_0 = \Der_\CC(\OO)$ are $(\CC,\OO)$-Lie-Rinehart algebras whose enveloping (or universal) algebras are respectively $\calV$ and $\DD$ (see Example \ref{ejem:app}). 
By the scalar extension $\CC \to \CC[s] $ we obtain the $(\CC[s],\OO[s])$-Lie-Rinehart algebras  $\calL:= L[s]$ and $\calL':= L'[s]=\Der_{\CC[s]}(\OO[s])$, and their enveloping algebras are respectively $\calV[s]$ and $\DD[s]$. Here, the filtrations in $\calV[s]$ and $\DD[s]$ are induced by the usual order of differential operators and, in both cases, the $0$-step of the filtration is $\OO[s]$. If we consider the ``total order filtration'' in both rings (the total order of $s$ is $1$), then they appear as the enveloping algebras of the $(\CC,\OO)$-Lie-Rinehart algebras $F^1\calV = \OO \oplus \fDer_\CC(-\log D)_0$ and $F^1\DD = \OO \oplus \Der_\CC(\OO)$ respectively.
\medskip

The rings  $\calV[s]$ and $\DD[s]$ are left and right Noetherian of finite global homological dimension (see Proposition \ref{prop:finite-ghd})
\medskip

The dualizing modules of $L'$ and $L$ (Definition \ref{def:dualizing}) are $\omega_{L'} = \bigwedge^d \Omega_{\CC^d,0}^1 = \Omega^d_{\CC^d,0}$ and $\omega_L = \bigwedge^d \Omega_{\CC^d,0}^1(-\log D)= \Omega_{\CC^d,0}^d(-\log D)$. By scalar extension, the dualizing modules of $\calL$ and $\calL'$ are
$\omega_{\calL} = \omega_L[s]$ and 
$\omega_{\calL'} = \omega_{L'}[s]$.
\medskip

If $\calM$ is a left $\DD[s]$-module (resp. a left $\calV[s]$-modules) free of finite rank over $\OO[s]$,
the left $\DD[s]$-module (resp. left $\calV[s]$-module) $\Hom_{\OO[s]}(\calM,\OO[s])$ (see \ref{internal-oper})
will be denoted by $\calM^*$.
\medskip

Let $\cd{b}{f}{\calV[s]}$ and $\cd{b}{f}{\DD[s]}$ be respectively the bounded derived categories of left $\calV[s]$-modules and of left $\DD[s]$-modules with finitely generated cohomologies. The following definition is a particular case of Definition \ref{def:dual}.

\begin{definition} 
The duality functors $\mathbb{V}: \cd{b}{f}{\calV[s]} \to \cd{b}{f}{\calV[s]}$ and $\Dual: \cd{b}{f}{\DD[s]} \to \cd{b}{f}{\DD[s]}$ are defined by
\begin{eqnarray*}
&
\mathbb{V} (\calM) = \left(\RR \Hom_{\calV[s]}(\calM, \calV[s])[d]\right)^{\text{\rm left}}, 
& 
\Dual (\calM) = \left(\RR \Hom_{\DD[s]}(\calM, \DD[s])[d]\right)^{\text{\rm left}}.
\end{eqnarray*}
\end{definition} 

The above functors are contravariant involutive triangulated functors. 

\begin{remark} \label{nota:referee}  Since $\CC[s]$ is contained in the center of the rings $\calV[s]$ and $\DD[s]$, the abelian categories of left (or right) modules over these rings are $\CC[s]$-linear, i.e. the set of morphisms between two objects is not only an abelian group, but also a $\CC[s]$-module, and the composition of morphisms is $\CC[s]$-bilinear. This enriched structure is inherited by the triangulated categories $\cd{b}{f}{\calV[s]}$ and $\cd{b}{f}{\DD[s]}$ and functors $\mathbb{V}$ and $\Dual$ are easily seen to be $\CC[s]$-linear.
\end{remark}

The following proposition is a particular case of Proposition \ref{prop:dual-connection}.

\begin{proposition} Let $\calM$ be a left $\calV[s]$-module (resp. a left $\DD[s]$-module). If $\calM$ is free of finite rank over $\OO[s]$, we have a canonical $\calV[s]$-isomorphism (resp. $\DD[s]$-isomorphism)
$ \mathbb{V} (\calM) \simeq \calM^*$ (resp. $\Dual(\calM) \simeq \calM^*$).
\end{proposition}

The above proposition gives for $\calM=\OO[s]$ canonical isomorphisms
$\mathbb{V} (\OO[s]) \simeq \OO[s]$, $\Dual(\OO[s]) \simeq \OO[s]$.
\medskip

Let us write $\OO(D)$ for the stalk at the origin of the integrable logarithmic connection $\OO_{\CC^d}(D)$ (see \cite[\S 1.2]{calde_nar_fourier}).

\begin{theorem} \label{teo:duality} For any complex $\calM$ in $\cd{b}{f}{\calV[s]}$ we have a canonical isomorphism in $\cd{b}{f}{\DD[s]}$
$$ \Dual\left( \DD[s]\Lotimes_{\calV[s]}  \calM \right) \simeq \DD[s]\Lotimes_{\calV[s]} \left(  \OO(D)[s]  \stackrel{\vphantom{L}}{\otimes}_{\OO[s]}    \mathbb{V}(\calM)\right).
$$
\end{theorem}

\begin{proof} It is a particular case of Theorem \ref{teo:duality-ferrara}. We only need to observe that the relative dualizing module is in this case
$\omega_{\calL'/\calL} =  \omega_{L'/L}[s] = \OO(D)[s] $   (see Example \ref{ejem:app-2}).
\end{proof}

\begin{corollary} \label{cor:dual-q-bernstein} For any polynomial $q(s)\in\CC[s]$, there is a canonical isomorphism
$$ \Dual\left( \DD[s]\Lotimes_{\calV[s]}  \OO[s]h^{q(s)} \right) \simeq \DD[s]\Lotimes_{\calV[s]} \OO[s]h^{-q(s)-1} .
$$
\end{corollary}

\begin{proof} It is a consequence of the canonical isomorphisms $\mathbb{V}\left(\OO[s]h^{q(s)} \right) \simeq
\left(\OO[s]h^{q(s)} \right)^*\simeq \OO[s]h^{-q(s)}$ and  
$\OO(D)[s]  \otimes_{\OO[s]} \OO[s]h^{-q(s)} \simeq 
\OO[s]h^{-q(s)-1}$.
\end{proof}

\begin{corollary} \label{cor:main} Under the above hypotheses, assume that our germ $h:(\CC^d,0) \to (\CC,0)$ satisfies the equivalent properties of Proposition \ref{prop:gen-LJT} and let $\varphi:\CC[s]\to \CC[s]$ be an automorphism of $\CC$-algebras. Then, there is a canonical isomorphism
$$ \Dual \left( \DD[s] h^{\varphi(s)} \right) \simeq \DD[s] h^{-\varphi(s)-1}.$$ 
\end{corollary}
\begin{proof} It is a consequence of Corollary \ref{cor:dual-q-bernstein} and Proposition \ref{prop:gen-LJT-varphi}.
\end{proof}

\section{The symmetry of Bernstein-Sato polynomials}

In this section we keep the notations of \S \ref{sec:dual}.

\begin{theorem} \label{teo:main-sym} Let $h:(\CC^d,0) \to (\CC,0)$ be a non-constant reduced germ of holomorphic function such that the divisor $D=h^{-1}(0)$ is free and satisfies the equivalent properties of Proposition \ref{prop:gen-LJT}. Then its Bernstein-Sato polynomial satisfies the equality
$b(s) = \pm b(-s-2)$.
\end{theorem}

\begin{proof}
Let us consider the exact sequence of left $\DD[s]$-modules
$$ 0 \to \DD[s] h^{s+1} \to \DD[s] h^s \to \calQ := \left(\DD[s] h^s\right) / \left(\DD[s] h^{s+1}\right) \to 0.$$
The Bernstein-Sato polynomial $b(s)$ of $h$ is by definition the minimal polynomial of the action of $s$ on $\calQ$.
By applying the duality functor $\Dual$ we obtain a triangle
$$ \Dual (\calQ ) \to \Dual \left( \DD[s] h^s \right) \to \Dual \left( \DD[s] h^{s+1} \right) \stackrel{+1}{\to}$$ and from Corollary \ref{cor:main} we deduce that the second arrow corresponds to the inclusion $\DD[s] h^{-s-1} \to \DD[s] h^{-s-2}$, $\Dual (\calQ ) $ is concentrated in degree $1$ and there is an exact sequence of left $\DD[s]$-modules
$$  0 \to \DD[s] h^{-s-1} \to \DD[s] h^{-s-2} \to \Dual^1 (\calQ ) \to 0.$$
Let us call $\varphi:\CC[s] \to \CC[s]$ the automorphism of $\CC$-algebras determined by $\varphi(s)=-s-2$. From Lemma \ref{lema:varphi} we deduce that 
$  \varphi^* \left( \calQ\right) \simeq \Dual^1 (\calQ ) 
$
and so the minimal polynomial of the action of $s$ on $\Dual^1 (\calQ ) $ is $\varphi(b(s)) = b(-s-2)$. 
\medskip

On the other hand, if we call $\mu_s: \calQ  \to \calQ $ and $\nu_s:\Dual^1 (\calQ ) \to \Dual^1 (\calQ )$ the actions of $s$, we have (see Remark \ref{nota:referee}) 
$b(\nu_s) = b(s\cdot \Id_{\Dual^1 (\calQ )} ) = b( \Dual^1(s\cdot \Id_\calQ) ) = b(\Dual^1 (\mu_s) ) = \Dual^1 (b(\mu_s))=0$.
We conclude that $b(s)$ is a multiple of $b(-s-2)$ and so $b(s) = \pm b(-s-2)$.
\end{proof} 

\begin{corollary} \label{cor:main-sym} Let $(D,0)\subset (\CC^d,0)$ be a germ of a free divisor of linear Jacobian type. Then its Bernstein-Sato polynomial satisfies  the equality
$b(s) = \pm b(-s-2)$.
\end{corollary}

\begin{corollary} Under the hypotheses of Theorem \ref{teo:main-sym}
the Bernstein-Sato polynomial of $h$ has no roots less or equal than $-2$.
\end{corollary}

\begin{remark} Let us notice that the above corollary implies that $-1$ is the only integer root of the Bernstein-Sato polynomial of $h$. It applies in particular to the case of (weakly) Koszul free divisors of differential linear type (or equivalently, to free divisors of linear Jacobian type; see Proposition \ref{prop:improve-torrelli}) and so it answers (partially) the question stated in \cite[Remark 4.7]{torrelli_2004}.
\end{remark}

\begin{corollary} \label{cor:application-LCT}
Let $D\subset X$ be a free divisor and assume that the equivalent properties of Proposition \ref{prop:gen-LJT} hold for some (and hence any) local reduced equation of $D$ at each point $p\in D$. Then, the canonical map
$\DD_X \Lotimes_{\calV_X} \OO_X(D) \to \OO_X[\star D]$
is an isomorphism and  the logarithmic comparison theorem holds.
\end{corollary}

\begin{proof} The problem being local, we can assume that $p=0\in\CC^d$ and $(D,0) \subset (\CC^d,0)$ is given by a reduced equation $h\in\OO$.
Since $-1$ is the smallest integer root of the Bernstein-Sato polynomial of $h$, we deduce that the $\DD$-module $\OO[\star D]$ is generated by $h^{-1}$ and that $\ann_{\DD} h^{-1}$ is obtained from $\ann_{\DD[s]} h^s$ by specializing $s=-1$ (cf. \cite[Proposition 3.1]{torrelli_2002}), and so $\ann_{\DD} h^{-1}$ is generated by order $1$ differential operators. In other words, the canonical map
$\DD \otimes_{\calV} \OO(D) \to \OO[\star D]$ is an isomorphism of left $\DD$-modules.
\medskip

In order to prove that the complex $\DD\Lotimes_{\calV} \OO(D)$ is concentrated in degree $0$, we proceed as in \cite[Proposition 2.2.17]{narvaez-contemp-2008}.
First, since $\OO[s] h^s$ has no $(s+1)$-torsion, we have 
$$\OO(D) = \OO h^{-1} \simeq \left(\calV[s]/\calV[s] (s+1) \right) \otimes_{\calV[s]} \OO[s] h^s \simeq \left(\calV[s]/\calV[s] (s+1) \right) \Lotimes_{\calV[s]} \OO[s] h^s,$$
and second
\begin{eqnarray*}
& \DD\Lotimes_{\calV} \OO(D) \simeq \cdots \simeq \DD\Lotimes_{\calV} \left(\calV[s]/\calV[s] (s+1) \right) \Lotimes_{\calV[s]} \OO[s] h^s \simeq& \\
& \left(\DD[s]/\DD[s] (s+1) \right) \Lotimes_{\calV[s]} \OO[s] h^s  \simeq \left(\DD[s]/\DD[s] (s+1) \right) \Lotimes_{\DD[s]} \DD[s] \Lotimes_{\calV[s]} \OO[s] h^s \simeq&\\
& \left(\DD[s]/\DD[s] (s+1) \right) \Lotimes_{\DD[s]}
\DD[s] h^s,
\end{eqnarray*}
but $\DD[s] h^s \subset \OO[s,h^{-1}] h^s$ has no $(s+1)$-torsion and so the last complex is concentrated in degree $0$.
\medskip

The last statement is a consequence of \cite[Theorem 4.1]{calde_nar_fourier}.
\end{proof}

\begin{remark} Let us notice that, after Corollary \ref{cor:LJT-properties}, the above corollary applies to free divisors of linear Jacobian type, and so to locally quasi-homogeneous free divisors \cite[Theorem 5.6]{calde_nar_compo}. In particular, it answers the missing point in \cite[Remark 1.25]{calde_nar_lct_ilc} and gives a purely algebraic proof of the logarithmic comparison theorem in \cite{cas_mond_nar_96}.
\end{remark}

The following result generalizes \cite[Theorem 1.6]{gran-schul-RIMS-2010} for any Koszul free divisor, non-necessarily reductive linear, and 
\cite[Proposition 2.3.1]{narvaez-contemp-2008} for higher dimension. It also improves \cite[Corollary 1.8]{torrelli_2004}.

\begin{theorem} Let $D\subset X$ be a Koszul free divisor. The following properties are equivalent:
\begin{enumerate}
\item[(a)] $D$ is strongly Euler homogeneous.
\item[(b)] $D$ is of linear Jacobian type.
\item[(c)] $D$ is strongly Koszul.
\item[(d)] $D$ is of differential linear type.
\item[(e)] $D$ satisfies the logarithmic comparison theorem.
\end{enumerate}
\end{theorem}

\begin{proof} (a) $\Rightarrow$ (b)\  We adapt the proof for locally quasi-homogeneous free divisors of \cite[Theorem 5.6]{calde_nar_compo} to the case of strongly Euler homogeneous free divisors. We proceed by induction on $\dim X$. For $\dim X = 2$ the result is known (cf. \cite[Proposition 2.3.1]{narvaez-contemp-2008}). Assume that the result is true whenever the ambient manifold has dimension $d-1$ and assume now that $\dim X = d \geq 3$.

Let $p\in D$ be a point. The question being local, we can assume that $X\subset \CC^d$ is a small open neighborhood of $p=0\in D$. Let $h:X\to\CC$ be a reduced equation of $D$
and $\calJ_D=(h,h'_{x_1},\dots,h'_{x_d})\subset \OO_X$ its Jacobian. Let $\{\delta_i = \sum_{j=1}^d a_{ij} \frac{\partial}{\partial
x_j}\}_{1\leq i \leq d}$ be a basis of $\Gamma(X,\fDer_\CC(-\log D))$. Since $D$ is strongly Euler homogeneous we can take $\delta_i(h)=0$ for all $i=1,\dots,d-1$, $\delta_d(h)=h$ and $a_{d1}(0)=\cdots=a_{dd}(0)=0$. In particular $h\in (h'_{x_1},\dots,h'_{x_d})$ and $\calJ_D=(h'_{x_1},\dots,h'_{x_d})$.

The kernel of the natural map
$\fDer_\CC(\OO_X) \to \calJ_D$ sending any derivation $\delta$ to $\delta(h)$ coincides with the free $\OO_X$-submodule $\widetilde{\Theta}_h \subset \fDer_\CC(-\log D)$ generated by $\delta_1,\dots,\delta_{d-1}$ and $\fDer_\CC(-\log D) = \widetilde{\Theta}_h \oplus \OO_X \delta_d$. Let us call 
$$\widetilde{\Phi}: \Sim \fDer_\CC(\OO_X) =\gr \DD_X =\OO_X[\xi] \to \Rees (\calJ_D), \quad 
\widetilde{\Phi}(\xi_j) = h'_{x_j} t,$$
the induced graded map, which is surjective.

Let us consider the augmented Koszul complex over $\gr \DD_X =\OO_X[\xi]$ 
associated with $\widetilde{\Theta}_h \equiv \sigma\left( \widetilde{\Theta}_h\right)  \subset  \gr^1 \DD_X$
\begin{eqnarray*}
&\displaystyle \mathbf{K}^\sbullet :=
0 \xrightarrow{}
\gr \DD_X\otimes_{\OO_X}\stackrel{d-1}{\bigwedge}\widetilde{\Theta}_h
\xrightarrow{d_{-d+1}} \cdots \xrightarrow{d_{-2}}
\gr \DD_X\otimes_{\OO_X}\stackrel{1}{\bigwedge}\widetilde{\Theta}_h
\xrightarrow{d_{-1}} &\\ & \xrightarrow{d_{-1}}\gr \DD_X \xrightarrow{d_0} 
\Rees (\calJ_D) \to 0, &\\ 
& \displaystyle d_{-k}(
F\otimes(\sigma_1\wedge\cdots\wedge\sigma_k))
 =\sum_{i=1}^k
(-1)^{i-1} F\sigma_i\otimes(\sigma_1\wedge\cdots\widehat{\sigma_i}
\cdots\wedge\sigma_k)&
\end{eqnarray*}
with augmentation $d_0=\widetilde{\Phi}$. Since $D$ is Koszul, $\sigma(\delta_1),\dots,\sigma(\delta_{d-1})$ is a regular sequence and so $\mathbf{K}^\sbullet$ is exact in degrees $\leq -1$. It is also exact in degrees $\geq 1$ because $\widetilde{\Phi}$ is surjective.
\medskip

In order to prove that $D$ is of linear Jacobian type we need to prove that $\mathbf{K}^\sbullet$  is exact in degree $0$, but for that it is enough to prove the exactness at any point $q\in X - \{0\}$. This reduction is given in \cite[Proposition 5.4]{calde_nar_compo} and we recall its proof, based on a local cohomology argument.
\medskip

The complex $\mathbf{K}^\sbullet$ is graded with $\mathbf{K}^{-j}_m = \gr^{m-j} \DD_X\otimes_{\OO_X}\stackrel{j}{\bigwedge}\widetilde{\Theta}_h$, $j=0,\dots,d-1$ and $\mathbf{K}^{1}_m = \calJ_D^m t^m$, for $m\geq 0$, and the reduction can be done on each degree. For each $m\geq 0$, let us consider the short sequence of coherent sheaves
\begin{equation} \label{eq:aux}
0 \xrightarrow{} \im d^m_{-1} \xrightarrow{} \mathbf{K}^{0}_m = \gr^{m} \DD_X \xrightarrow{d^m_{0}} \mathbf{K}^{1}_m=\calJ_D^m t^m
\xrightarrow{} 0
\end{equation}
and assume it is exact on $ X - \{0\}$.
Since $H^i_0(\OO_X)=0$ for $i\neq d$ and the complex $\mathbf{K}^\sbullet$ is exact in degrees $\leq -1$, we deduce that $H^1_0(\im d^m_{-1})=0$ . On the other hand, the sheaves $\im d^m_{-1}$ and $\mathbf{K}^{0}_m $ have no sections supported by the origin because both are subsheaves of locally free $\OO_X$-modules. By using that $d^m_{0}$ is surjective and the long exact sequences associated with local cohomology, we deduce that (\ref{eq:aux}) is exact everywhere (cf. \cite[(8.14)]{loo_84}).
\medskip

Let us now prove that $\mathbf{K}^\sbullet$ is exact on $ X - \{0\}$. The exactness on $X-D$ is clear. Let $q\in D - \{0\}$ be a point. From the Koszul hypothesis we know (cf. \cite[Corollary 1.9]{calde_nar_compo}) that the zero locus of the symbols $\sigma(\delta_i) = \sum_{j=1}^d a_{ij} \xi_j$, $i=1,\dots,d$, in the cotangent bundle $T^* X$ has dimension $d$, and so the zero locus of the coefficients $a_{ij}$, $i,j=1,\dots,d$, is the origin and there is a logarithmic derivation with respect to $D$ 
non-vanishing at $q$. We can integrate it and deduce that $(X,D,p)$ is analytically isomorphic to $(\CC^{d-1} \times \CC,D'\times \CC,(0,0))$, where $(D',0)\subset (\CC^{d-1},0)$ is a germ of hypersurface. It is easy to see that $(D',0)$ is again a germ of Koszul free divisor (cf. \cite[Proposition 1.10]{calde_nar_compo}) and strongly Euler homogeneous (cf. Remark \ref{nota:strong-E-LJT-factor}). By the induction hypothesis, we deduce that $(D',0)$ is of linear Jacobian type. So $D$ is of linear Jacobian type at $q$ and $\mathbf{K}^\sbullet$  is exact in degree $0$ at $q$.
\medskip

\noindent Equivalences (b) $\Leftrightarrow$ (c), (c) $\Leftrightarrow$ (d) have been proven respectively in Proposition \ref{prop:LJT-SK} and Proposition \ref{prop:improve-torrelli}.
\medskip

\noindent (d) $\Rightarrow$ (e)\ It is a consequence of Corollary \ref{cor:application-LCT} and Corollary \ref{cor:LJT-properties}.
\medskip

\noindent (e) $\Rightarrow$ (a)\  We proceed by induction on $\dim X$. If $\dim X= 2$ we know from \cite{calde_mon_nar_cas} that $D$ is locally quasi-homogeneous and so it is strongly Euler homogeneous. Assume that the implication (e) $\Rightarrow$ (a) is true whenever the dimension of the ambient manifold is $d-1$ and assume now that $\dim X = d \geq 3$. From \cite[Corollary 4.3]{calde_nar_fourier} or \cite[Criterion 4.1]{cas_ucha_exper} we deduce that the differential operators annihilating $h^{-1}$, where $h$ is a reduced equation of our divisor $D$, are generated by order one operators, and from \cite[Corollary 1.8]{torrelli_2004} we conclude that $D$ is Euler homogeneous. It remains to prove that for any point $p\in D$ there is an Euler vector field with respect to some (and hence, for any) reduced equation of $(D,p)$, vanishing on $p$. 
Let us take a such reduced equation $h=0$ and an Euler vector field $\chi$ on a neighborhood of $p$ such that $\chi(h)=h$. If $\chi$ vanishes at $p$, we are done. If not, we can integrate $\chi$ and deduce that $(X,D,p)$ is analytically isomorphic to $(\CC^{d-1} \times \CC,D'\times \CC,(0,0))$, where $(D',0)\subset (\CC^{d-1},0)$ is a germ of hypersurface. We deduce as before that $(D',0)$ is again a germ of Koszul free divisor, and from Proposition \ref{prop:LCT-X-XxC}, $D'$ satisfies the logarithmic comparison theorem. By the induction hypothesis, $D'$ is strongly Euler homogeneous and so $D$ is also strongly Euler homogeneous.
\end{proof}

Here we include a well known result, but for which we have not found an explicit statement in the literature.

\begin{proposition} \label{prop:LCT-X-XxC} Let $D \subset X$ be a divisor. The following properties are equivalent:
\begin{enumerate}
\item[(a)] The logarithmic comparison theorem holds for $D\subset X$.
\item[(b)] The logarithmic comparison theorem holds for $D \times \CC \subset X \times \CC$.
\end{enumerate}
\end{proposition}

\begin{proof} Let us call $\widetilde{X} = X \times \CC$, $\pi: \widetilde{X} \to X$ the projection and $\widetilde{D} = D \times \CC = \pi^{-1}(D)$. We proceed as in \cite[Lemma 2.2]{cas_mond_nar_96}. There are natural inclusions 
$$ \pi^{-1} \Omega^{\bullet}_X(\log D) \subset \Omega^{\bullet}_{\widetilde{X}}(\log \widetilde{D}),\ \pi^{-1} \Omega^{\bullet}_X(\star D) \subset \Omega^{\bullet}_{\widetilde{X}}(\star \widetilde{D})$$
which are quasi-isomorphisms. So, the following properties are equivalent:
\begin{enumerate}
\item[(a)] The logarithmic comparison theorem holds for $D\subset X$, i.e the inclusion $\Omega^{\bullet}_X(\log D) \to \Omega^{\bullet}_X(\star D)$ is a quasi-isomorphism.
\item[(a')] The inclusion $\pi^{-1}\Omega^{\bullet}_X(\log D) \to \pi^{-1}\Omega^{\bullet}_X(\star D)$ is a quasi-isomorphism.
\item[(b)] The logarithmic comparison theorem holds for $D \times \CC \subset X \times \CC$, i.e. the inclusion $\Omega^{\bullet}_{\widetilde{X}}(\log \widetilde{D}) \to \Omega^{\bullet}_{\widetilde{X}}(\star \widetilde{D})$ is a quasi-isomorphism.
\end{enumerate}
\end{proof}

\section{The case of logarithmic connections}

In this section we generalize the preceding results to the case of integrable logarithmic connections. Proofs remain essentially the same and will be only sketched. This generalization is based on \cite[Theorem 3.1, Corollary 3.2]{calde_nar_lct_ilc}, and so we will assume that our germ of free divisor $(D,0) \subset (\CC^d,0)$, with reduced equation $h:(\CC^d,0) \to (\CC,0)$, is of linear Jacobian type.
We keep the notations of \S \ref{sec:log-mero} and \S \ref{sec:dual}.
\medskip

Let $\EE$ be
a germ at $0$ of integrable logarithmic connection with respect $D$, i.e. $\EE$ is a free $\OO$-module of finite rank endowed with a left $\calV$-module structure. As explained in \cite[\S\ 3.1]{calde_nar_lct_ilc}, we consider the logarithmic Bernstein-Kashiwara module $\EE[s] h^s$ inside the meromorphic connection $\EE[s,h^{-1}] h^s$. Corollary 3.2 in \cite{calde_nar_lct_ilc} tells us that the canonical map
$$ \DD_X[s]\Lotimes_{\calV_X[s]} \EE[s] h^s \to \DD_X[s]\EE[s] h^s
$$
is an isomorphism in $\cd{b}{f}{\DD[s]}$.
\medskip

More generally, for any $q(s)\in\CC[s]$ we consider the {\em $q(s)$-Bernstein module} associated with $\EE$
$$\EE[s] h^{q(s)} := \EE[s] \otimes_{\OO[s]} \OO[s]h^{q(s)}   \subset \EE[s,h^{-1}] h^{q(s)}:= \EE[s,h^{-1}] \otimes_{\OO[s]} \OO[s,h^{-1}] h^{q(s)} $$  
as in Definition \ref{def:q-bernstein}. In the same way as in Lemma \ref{lema:varphi} and Proposition \ref{prop:gen-LJT-varphi} we prove that for any automorphism of $\CC$-algebras $\varphi:\CC[s] \to \CC[s]$ the canonical map
\begin{equation} \label{eq:ilc-varphi}
\varphi^* \left( \DD[s] \EE[s]h^{q(s)} \right) := \DD[s] \otimes_{\varphi} \left( \DD[s] \EE[s]h^{s}  \right) \to \DD[s] \EE[s] h^{\varphi(q(s))}
\end{equation}
induced by $\overline{\varphi}: \OO[s,h^{-1}] h^{q(s)} \to \OO[s,h^{-1}] h^{\varphi(q(s))} $
is an isomorphism of left $\DD[s]$-modules. Also, the canonical map
$$ \DD[s] \Lotimes_{\calV[s]} \left( \EE[s] h^{\varphi(s)}  \right) \xrightarrow{} \DD[s] \EE[s] h^{\varphi(s)} $$ 
is an isomorphism in the derived category of left $\DD[s]$-modules. 
As in Corollary \ref{cor:main} we obtain a canonical isomorphism
$$ \Dual \left( \DD[s] \EE[s]h^{\varphi(s)} \right) \simeq \DD[s] \EE^*[s]h^{-\varphi(s)-1}.$$ 
Recall that the Bernstein-Sato polynomial of $\EE$ is defined as the minimal polynomial of the action of $s$ on the quotient
$\calQ_{\EE}:=\DD[s] \EE[s]h^{s}/ \DD[s] \EE[s]h^{s+1}$ and it is denoted by $b_{\EE}(s)$ \cite[Remark 3.5]{calde_nar_lct_ilc}. The existence of a non-zero $b_{\EE}(s)$ is a straightforward consequence of the existence of non-trivial Bernstein-Sato functional equations with respect to sections of holonomic $\DD$-modules 
(\cite[Theorem 2.7]{kashiwara-1978}; see also \cite{meb_nar_dmw}).
\medskip

Now we are ready to state and prove the announced extension of Theorem \ref{teo:main-sym} to the case of arbitrary logarithmic connections.

\begin{theorem} \label{teo:main-2} Let $(D,0) \subset (\CC^d,0)$ be a germ of a free divisor of linear Jacobian type with reduced equation $h:(\CC^d,0) \to (\CC,0)$ and 
$\EE$ a germ at $0$ of integrable logarithmic connection with respect to $D$. Then the Bernstein-Sato polynomials of $\EE$ and of its dual $\EE^*$ are related by the equality
$$ b_{\EE}(s) = \pm b_{\EE^*}(-s-2).$$
\end{theorem}

\begin{proof} We proceed as in the proof of Theorem \ref{teo:main-sym}. Let us consider the exact sequence of left $\DD[s]$-modules
$$ 0 \to \DD[s] \EE[s] h^{s+1} \to \DD[s] \EE[s] h^s \to \calQ_{\EE}\to 0.$$
By applying the duality functor $\Dual$ we obtain a triangle
$$ \Dual (\calQ_{\EE}) \to \Dual \left( \DD[s] \EE[s] h^s \right) \to \Dual \left( \DD[s] \EE[s] h^{s+1} \right) \stackrel{+1}{\to}$$ in which the second arrow corresponds to the inclusion $\DD[s] \EE^*[s] h^{-s-1} \to \DD[s] \EE^*[s]h^{-s-2}$, $\Dual (\calQ_{\EE} ) $ is concentrated in degree $1$ and there is an exact sequence of left $\DD[s]$-modules
$$  0 \to \DD[s] \EE^*[s]  h^{-s-1} \to \DD[s] \EE^*[s]  h^{-s-2} \to \Dual^1 (\calQ_{\EE} ) \to 0.$$
Let $\varphi:\CC[s] \to \CC[s]$ be the automorphism of $\CC$-algebras determined by $\varphi(s)=-s-2$ and let us consider the exact 
sequence of left $\DD[s]$-modules
$$ 0 \to \DD[s] \EE^*[s] h^{s+1} \to \DD[s] \EE^*[s] h^s \to \calQ_{\EE^*}\to 0.$$
From (\ref{eq:ilc-varphi}) we deduce an isomorphism $  \varphi^* \left( \calQ_{\EE^*}\right) \simeq \Dual^1 (\calQ_{\EE}) $ and so the minimal polynomial of the action of $s$ on $\Dual^1 (\calQ_{\EE} ) $ is $\varphi(b_{\EE^*}(s)) = b_{\EE^*}(-s-2)$. On the other hand, the action of $s$ on $\Dual^1 (\calQ_{\EE} )$ is annihilated by $b_{\EE}(s)$ and we conclude that $b_{\EE}(s)$ is a multiple of $b_{\EE^*}(-s-2)$.
\medskip

In a symmetric way we deduce that $b_{\EE^*}(s)$ is a multiple of $b_{\EE}(-s-2)$, or equivalently $b_{\EE^*}(-s-2)$ is a multiple of $b_{\EE}(s)$, and so $ b_{\EE}(s) = \pm b_{\EE^*}(-s-2)$.
\end{proof}

\section{Open questions}

The starting point of this paper has been \cite{gran-schul-RIMS-2010} and the observed effect of duality on the Bernstein-Sato polynomial of some examples of integrable logarithmic connections with respect to quasi-homogeneous plane curves 
\cite{narvaez-ESI-2011}.
\medskip

In \cite{gran-schul-RIMS-2010}, the authors proved that the Bernstein-Sato polynomial of any regular special linear free divisor (in particular, any reductive linear free divisor), and of any reductive prehomogeneous determinant (these are the non-reduced version of linear free divisors) have the symmetry property $b(s)=\pm b(-s-2)$ (see Theorems 3.5 and 5.5 and the definition of the involved notions in \cite{gran-schul-RIMS-2010}). These results and our theorem \ref{teo:main-sym} overlap, but they are logically independent. On one hand, the results in \cite{gran-schul-RIMS-2010} only apply to invariants of some prehomogeneous vector spaces. On the other hand, our theorem \ref{teo:main-sym} cannot cover either the case of non-reduced reductive prehomogeneous determinants, since it only applies to reduced equations, or the case of reductive linear free divisor, since there are examples of such divisors which are not of differential linear type. Namely, D. Andres and J. Mart\'{\i}n-Morales have recently studied the example $D=\{h=0\} \subset M_{3,4}= \CC^{3 \times 4}$ in \cite[Proposition 7.12]{gran-mon-nieto-schul} by using the techniques in \cite{andres_PhD} and have found differential operators of order $2$ annihilating $h^s$ which are not in $\ann^{(1)}_{\DD[s]} h^s = \DD[s] \ann_{\calV [s]} h^s$.
\medskip

All the examples of free divisors given in \cite{nakayama_sekiguchi} are of linear Jacobian type and so they are covered by Theorem \ref{teo:main-sym}.

\begin{question} Is there a common generalization of Theorem \ref{teo:main-sym} and Theorems 3.5 and 5.5 in \cite{gran-schul-RIMS-2010}?
\end{question}

\begin{question} \label{question:sym-gen} 
The reduced Bernstein-Sato polynomial $\widetilde{b}(s) = \frac{b(s)}{s+1}$ of any quasi-homogeneous polynomial $h:\CC^d\to \CC$ with an isolated singularity at the origin satisfies the symmetry $\widetilde{b}(s)=\pm \widetilde{b}(-s-d)$. 
This is a consequence of a celebrated result of Malgrange \cite{mal_75}, namely that
the roots of $\widetilde{b}(s)$ for any isolated singularity are the eigenvalues of the connection on the saturated
Brieskorn lattice which equals the usual one precisely in the quasi-homogeneous case. In this case 
the spectrum and the roots of $\widetilde{b}(s)$ are (up to an overall shift) the same, hence
the symmetry of the spectrum gives the symmetry of the roots of $\widetilde{b}(s)$. 
\smallskip

In the non-quasi-homogeneous isolated singularity case, the spectrum is different from the roots of $\widetilde{b}(s)$, so that the
symmetry of the spectrum does not imply the symmetry of the roots of $\widetilde{b}(s)$.
\smallskip

Let us notice that in the quasi-homogeneous isolated singularity case, the roots of $\widetilde{b}(s)$ can be explicitly listed in terms of the weights of variables (cf. \cite[Corollary 3.9]{yano_1978}) and the symmetry can be checked directly. Also, R. Arcadias and the author observed that this symmetry can be obtained by using D-module duality theory.
\smallskip

The intersection of Theorem \ref{teo:main-sym} and the symmetry property for quasi-homogeneous isolated singularities is the case of quasi-homogeneous plane curves. Quasi-homogeneous isolated singularities are of linear Jacobian type and their Jacobian ideal are a complete intersection, and so Cohen-Macaulay. On the other hand, the Jacobian ideal of a singular free divisor is also Cohen-Macaulay of codimension $2$. 
\smallskip

By means of M. Saito's formula for the reduced $b$-function of the Thom-Sebastiani join of two germs, we can give, for any $e=2,\dots, d$, non-trivial examples of irreducible germs $h:(\CC^d,0) \to (\CC,0)$ such that the following properties hold: (a) the Jacobian ideal $J$ is of linear type; (b) 
$J$ is Cohen-Macaulay; and (c) the equality $\widetilde{b}_h(s)=\pm \widetilde{b}_h(-s-e)$ holds for $e=$ codimension of the singular locus of $\{h=0\}$. Nevertheless, (a) + (b) $\not\Rightarrow$ (c) as shown in Example \ref{ejemplo:det}.
\smallskip

The question of finding general criteria on $h$ implying property (c) and extending the known extreme cases of quasi-homogeneous isolated singularities ($e=d$) and free divisors of linear Jacobian type ($e=2$) seems interesting and remains open.
\end{question}

I owe Kari Vilonen the suggestion of checking the following example.

\begin{example}  \label{ejemplo:det} Let  $X$ be the vector space of square $n\times n$ complex matrices and $h=\det : X\to \CC$. It is well known that the Bernstein-Sato polynomial of $h$ is
 given by $ b_h(s) = (s+1)(s+2)\cdots (s+n)$ and so $ \widetilde{b}_h(s) = (s+2)\cdots (s+n)$. The Jacobian ideal $J$ of $h$ is generated by all the $(n-1)\times (n-1)$ minors of the generic matrix $(x_{ij})$. In particular the singular locus of $D=h^{-1}(0)$ consists of matrices of rank $\leq n-2$,
 $\dim D^{\text{sing}} = n^2 - 4$ and $e=\codim D^{\text{sing}} = 4$. We know that $J$ is of linear type \cite{huneke-86} and Cohen-Macaulay 
 \cite[page 25]{bruns-vetter-1988}. However, the equality $\widetilde{b}_h(-s-4) =\pm \widetilde{b}_h(s)$ only holds for $n=2$.
\end{example}

\begin{question}  Corollary \ref{cor:main-sym} applies to locally quasi-homogeneous free divisors after \cite[Theorem 5.6]{calde_nar_compo}), in particular to free hyperplane arrangements. However, the Bernstein-Sato polynomial of a non-free hyperplane arrangement does not satisfy in general the symmetry $b(s) = \pm b(-s-2)$. For instance, for  $h=x_1 x_2 x_3 (x_1+x_2+x_3)=0$ we have $b_h(s)= (s+1)^3(s+3/2)(s+3/4)(s+5/4)$, and  $b_h(s) \neq \pm b_h(-s-2)$. In fact, this symmetry property fails in many other examples of non-free hyperpane arrangements. Can we characterize the hyperplane arrangements whose Bernstein-Sato polynomial satisfy the above symmetry property?
\end{question}

\section*{Appendix A}

\setcounter{numero}{0}
\renewcommand{\thenumero}{(A.\arabic{numero})}

This appendix contains some basic notions and results about Lie-Rinehart algebras and their modules, a simple proof of the associativity law in \cite{calde_nar_ferrara} (see Theorem \ref{teoapp:main-1}) and a detailed proof of the duality theorem in  \cite{calde_nar_ferrara} (see Theorem \ref{teo:duality-ferrara}). To be brief, 
we have included only the statements and results strictly needed for the proof of the cited results and of Propositions \ref{prop:nuevomain-2} and \ref{prop:nuevomain-3}. There are variants of all these results which are left up to the reader (see Remark \ref{variants}).
\medskip

Let $k \to A$ be a homomorphism of commutative rings. Let us denote by $\Der_k(A)$ the $A$-module of $k$-linear
derivations $\lambda:A\to A$, which is a left sub-$A$-module of
$\End_k(A)$ closed by the bracket $[-,-]$. 
\medskip

A Lie-Rinehart algebra over $(k,A)$ (or a $(k,A)$-Lie algebra in \cite{rine-63}) is
an $A$-module $L$ endowed with a $k$-Lie algebra structure and an
$A$-linear map $\rho: L \to \Der_k(A)$, called {\em anchor map}, which is also a morphism of
Lie algebras and satisfies
$$ [\lambda,a\lambda'] = a[\lambda,\lambda'] + \rho(\lambda)(a)\lambda'$$
for $\lambda,\lambda'\in L$ and $a\in A$. To simplify, we write
$\lambda(a) := \rho(\lambda)(a)$ for
$\lambda\in L$ and $a\in A$.
\medskip

Let $(L,\rho),(L',\rho')$ be two Lie-Rinehart algebras over $(k,A)$. A map of Lie-Rinehart algebras $f:L\to L'$ is an $A$-linear map which is also a $k$-Lie algebra map and such that $\rho'\pcirc f= \rho$.
\medskip

It is  clear that $\Der_k(A)$ is a Lie-Rinehart algebra over $(k,A)$ with the identity as anchor map, and that for any Lie-Rinehart algebra $(L,\rho)$, $\rho$ is a map of Lie-Rinehart algebras.
\medskip

Let $L$ be a Lie-Rinehart algebra over $(k,A)$. A left $L$-module is an $A$-module $M$ endowed with
a $k$-bilinear action $(\lambda,m) \in L\times M \to \lambda m \in M$ such that 
$$(a\lambda)m = a(\lambda m),\  [\lambda,\lambda']m = \lambda(\lambda' m)-\lambda'(\lambda m),\ \lambda (am) = a(\lambda m) + \lambda(a) m$$
for all $a\in A$, $\lambda,\lambda'\in L$ and $m\in M$.
\smallskip

The $A$-module $A$ becomes a left $L$-module with the action $(\lambda,a)\in L\times A \mapsto \lambda(a)\in A$.
\smallskip

A right $L$-module is an $A$-module $Q$, where the multiplication by elements of $A$ is written on the right, endowed with a $k$-bilinear action $(q,\lambda) \in Q\times L \to q\lambda \in Q$ such that 
$$q(a\lambda) = (qa)\lambda,\  q[\lambda,\lambda'] = (q\lambda)\lambda'-(q\lambda')\lambda,\ 
(qa)\lambda  = (q\lambda)a - q\lambda(a)$$
for all $a\in A$, $\lambda,\lambda'\in L$ and $q\in Q$. 
\medskip

Let $L$ be a Lie-Rinehart algebra over $(k,A)$ and $U=\U(L)$ its {\em enveloping} (or {\em universal})
algebra (see \cite{rine-63}). It is endowed with an injective ring map $A\simeq U_0 \hookrightarrow U$ with $k \to U$ central, and a left $A$-linear map $L \to U$. Moreover, $U$ is generated as a ring by $A$ and the image of $L$, and it carries a canonical filtration $(U_r)_{r\geq 0}$ ($U_r$ is generated as left or right $A$-module by all products of elements of $L$ of length $\leq r$) such that $\gr U$ is a commutative $A$-algebra. 
\smallskip

From the universal property of $U$ (see \cite[pgs. 63-64]{MR1058984}) we have the following:
If $M$ is an $A$-module, to give a left $L$-module structure on $M$ is equivalent to extending its $A$-module structure to a left $U$-module structure.
Similarly, if $Q$ is an $A$-module, to give a right $L$-module structure on $Q$ is equivalent to extending its $A$-module structure to a right $U$-module structure.
\smallskip

The surjection $p\in U \mapsto p\cdot 1 \in A$ induces a canonical isomorphism of left $U$-modules 
$ U/U\cdot L  \simeq A$.
\medskip

\begin{example} \label{ejem:app}  (a) If $X$ is a complex smooth manifold, $p\in X$, $k=\CC$ and $A= \OO_{X,p}$ is the ring of germs at $p$ of holomorphic functions, then the enveloping algebra of the Lie-Rinehart algebra $\Der_\CC(\OO_{X,p})$ is the ring $\DD_{X,p}$ of germs at $p$ of linear differential operators with holomorphic coefficients in $X$.
\smallskip

\noindent (b) If $X$ is a complex smooth manifold, $D\subset X$ is a free divisor, $p\in D$, $k=\CC$, $A= \OO_{X,p}$ and $L=\fDer(-\log D)_p$ is the Lie-Rinehart algebra of germs of logarithmic vector fields with respect to $D$, then the enveloping algebra of $L$ is the ring of germs of logarithmic differential operators $\DD_{X,p}(-\log D)$, which coincides with the subring of $\DD_{X,p}$ generated by $\OO_{X,p}$ and $\fDer(-\log D)_p$  (see \cite[prop. 2.2.5]{calde_ens}).
\end{example}

\numero {\em Internal operations:} \label{internal-oper}   In what follows, $M, M_1, M_2, \dots $ will denote left $U$-modules and $Q, Q_1, Q_2, \dots $ right $U$-modules.
\smallskip It is well known that the $A$-modules $M_1\otimes_A M_2$, $\Hom_A(M_1,M_2)$ and $\Hom_A(Q_1,Q_2)$ (resp. $Q_1\otimes_A M_1$ and $\Hom_A(M_1,Q_1)$) have natural left (resp. right) $U$-module structures (cf. \cite[\S 2]{MR2000f:53109}).
\medskip

\inhibe
\noindent The $A$-module $M_1\otimes_A M_2$ has natural left $U$-module structure given by
 \begin{equation*} 
 \lambda(m_1\otimes m_2) = (\lambda m_1)\otimes m_2 + m_1\otimes (\lambda
m_2),\quad m_i\in M_i, i=1,2, \lambda\in L.
\end{equation*}
\noindent The $A$-module $\Hom_A(M_1,M_2)$ has a natural left $U$-module structure given by:
\begin{equation*}
(\lambda h)(m) = -h(\lambda m) + \lambda h(m), \quad h\in
\Hom_A(M_1,M_2), m\in M_1, \lambda\in L.
\end{equation*}
\noindent The $A$-module $Q_1\otimes_A M_1$ has a natural right $U$-module structure given by:
\begin{equation*}
(q\otimes m)\lambda = (q\lambda)\otimes m - q\otimes (\lambda
m),\quad m\in M_1, q\in Q_1, \lambda\in L.
\end{equation*}
\noindent The $A$-module $\Hom_A(M_1,Q_1)$ has a natural right $U$-module structure given by:
\begin{equation*} 
(h\lambda)(m) = h(\lambda m) +  h(m)\lambda, \quad h\in
\Hom_A(M_1,Q_1), m\in M_1, \lambda\in L.
\end{equation*}
\noindent The $A$-module $\Hom_A(Q_1,Q_2)$ has a natural left $U$-module structure given by:
\begin{equation*} 
(\lambda h)(q) = h(q\lambda ) - h(q)\lambda, \quad h\in
\Hom_A(Q_1,Q_2), q\in Q_1, \lambda\in L.
\end{equation*}
\endinhibe

\numero  \label{internal-oper-nat-iso} The natural $A$-linear maps
$ M_1 \otimes_A M_2 \simeq M_2 \otimes_A M_1,
 A\otimes_A M_2 \simeq M_2 \simeq \Hom_A(A,M_2),
[M_1 \otimes_A M_2] \otimes_A M_3 \simeq M_1 \otimes_A [M_2 \otimes_A M_3]$
are left $U$-linear, and the natural $A$-linear maps
$ Q_1\otimes_A A \simeq Q_1 \simeq \Hom_A(A,Q_1),
[Q_1\otimes_A M_1] \otimes_A M_2 \simeq Q_1 \otimes_A [M_1 \otimes_A M_2]$
are right $U$-linear.
\medskip

The proof of the following lemmas is straightforward. 

\begin{lemma}  \label{lema:varios-isos} The natural isomorphisms of $A$-modules 
\begin{eqnarray} \label{eq:f1}
&\Hom_A(M_1\otimes_A
M_2,M_3) \simeq \Hom_A(M_1,\Hom_A(M_2,M_3)),&\\
\label{eq:f2} &\Hom_A(Q_1\otimes_A M_1,Q_2) \simeq \Hom_A(M_1,\Hom_A(Q_1,Q_2)) &
\end{eqnarray}
are left $U$-linear.
\end{lemma}

\begin{lemma}  \label{lema:varios-eval-hom} The natural $A$-linear map
$Q_1 \otimes_A \Hom_A(Q_1,Q_2)  \to Q_2$
is right $U$-linear.
Moreover, it is an isomorphisms if $Q_1$ is a free $A$-module of rank 1.
\end{lemma}

If $M_1$ is a projective $A$-module of finite rank, we denote $M_1^*=\Hom_A(M_1,A)$.

\begin{lemma} \label{coro:varios-can-hom} Assume that $M_1$ is a projective $A$-module of finite rank. Then, the natural map
$  M_1^*\otimes_A M_2 \to \Hom_A(M_1,M_2)$ is an isomorphism of left $U$-modules.
\end{lemma}

\begin{lemma}  \label{lema:varios-can-hom-2} The natural $A$-linear map
$M_1  \to \Hom_A(Q_1,Q_1\otimes_A M_1)$
is left $U$-linear.
Moreover, it is an isomorphism if $Q_1$ is a free $A$-module of rank 1.
\end{lemma}

\begin{lemma} \label{lema:U-A} For any left $U$-modules $M$ and $M'$ (resp. for any right $U$-modules $Q$ and $Q'$),
 a map $h\in \Hom_A(M,M')$ (resp. $h\in\Hom_A(Q,Q')$) is $U$-linear if and only if $\lambda h =0$ for all $\lambda \in L$. Consequently 
there are canonical isomorphisms   $\Hom_U(M,M') \simeq \Hom_U(A, \Hom_A(M,M'))$, $
\Hom_U(Q,Q') \simeq \Hom_U(A, \Hom_A(Q,Q'))$.
\end{lemma}

\begin{lemma} \label{lema:tecnico-1} The $U$-linear isomorphism (\ref{eq:f1})
in Lemma
 \ref{lema:varios-isos} and Lemma \ref{lema:U-A} induce a $k$-linear isomorphism
 $$ \beta: \Hom_U(M_1\otimes_A
M_2,M_3) \xrightarrow{} \Hom_U(M_1,\Hom_A(M_2,M_3)).$$
\end{lemma}

\begin{lemma} \label{lema:tecnico-2} The $U$-linear isomorphism (\ref{eq:f2})
 in Lemma
 \ref{lema:varios-isos} and Lemma \ref{lema:U-A} induce a $k$-linear isomorphism
 $$ \gamma: \Hom_U(Q_1\otimes_A M_1,Q_2) \xrightarrow{\simeq} \Hom_U(M_1,\Hom_A(Q_1,Q_2)).$$
\end{lemma}

Assume now that 
$L_0 \to L$ is 
a map of Lie-Rinehart algebras over $(k,A)$, which induces a
ring map $U_0=\U(L_0)\to U=\U(L)$ between their enveloping
algebras. 
\medskip

Given a right $U$-module $Q$ and a left
$U_0$-module $M_0$, we know that $Q\otimes_A M_0$ and $Q \otimes_A
[U\otimes_{U_0} M_0]$ have natural right module structures over
$U_0$ and $U$, respectively, and that the map 
$$\sigma:
Q\otimes_A M_0 \to Q\otimes_A[U\otimes_{U_0} M_0],\quad \sigma(q\otimes
m) = q\otimes [1\otimes m]$$
is $U_0$-linear.
\medskip

The following theorem gives the associativity law needed in the proof of the duality theorem \ref{teo:duality-ferrara}. The proof we give here simplifies the original proof in \cite{calde_nar_ferrara}.

\begin{theorem} \label{teoapp:main-1} 
The $U$-linear map
$$\tau: [Q\otimes_A M_0]\otimes_{U_0} U \to
Q\otimes_A[U\otimes_{U_0} M_0]$$induced by $\sigma$ is an isomorphism
of right $U$-modules.
\end{theorem}

\begin{proof} It is enough to prove that for any right $U$-module $Q'$, the induced map
$$ \tau^*: \Hom_{U}(Q\otimes_A[U\otimes_{U_0} M_0],Q') \xrightarrow{} \Hom_{U}([Q\otimes_A M_0]\otimes_{U_0} U,Q')$$
is an isomorphism, and for that we consider the following commutative diagram 

\begin{equation*}
\begin{CD}
\Hom_U([Q\otimes_A M_0]\otimes_{U_0} U,Q') @<{\tau^*}<<
\Hom_U(Q\otimes_A[U\otimes_{U_0} M_0],Q')\\
@V{\alpha_1^*}V{\simeq}V @V{\gamma}V{\simeq\ \text{by \ref{lema:tecnico-2}}}V\\
\Hom_{U_0}(Q\otimes_A M_0,Q')  @. \Hom_U(U\otimes_{U_0}
M_0,\Hom_A(Q,Q'))\\
@V{\simeq\ \text{by \ref{lema:tecnico-2}}}V{\gamma_0}V @V{\simeq}V{\alpha_2^*}V\\
\Hom_{U_0}(M_0,\Hom_A(Q,Q')) @= \Hom_{U_0}(M_0,\Hom_A(Q,Q')).
\end{CD}
\end{equation*}
where 
 $\alpha_1:Q\otimes_A M_0 \xrightarrow{} [Q\otimes_A M_0]\otimes_{U_0} U$ is the natural (right) $U_0$-linear map defined by $\alpha_1(\xi ) = \xi\otimes 1$ for $\xi \in Q\otimes_A M_0$ and 
$\alpha_2:M_0 \xrightarrow{} U\otimes_{U_0} M_0$ is the natural (left) $U_0$-linear map defined by $\alpha_2(m) = 1\otimes m$ for $m \in M_0$.
\end{proof}

Let us notice that $Q\otimes_A U$ has two right $U$-module structures: the first one comes by scalar extension $A\to U$ from the $A$-module structure on $Q$ (we forget here the right $U$-module structure on $Q$):  
$$(q \otimes p)\pcirc p' := q\otimes (pp'),\quad q\in Q, p, p'\in U,$$
and the second one comes by \ref{internal-oper} from the right $U$-module structure on $Q$ and the left $U$-module structure on $U$:
$$ (q \otimes p)\star a := (qa) \otimes p = q \otimes (ap), \quad q\in Q, p\in U, a\in A,$$
$$ (q \otimes p)\star \lambda := (q\lambda)\otimes p - q\otimes (\lambda p),\quad q\in Q, p\in U, \lambda\in L.$$
It is clear that $(\xi \pcirc p)\star a = (\xi \star a)\pcirc p$, $(\xi\pcirc p)\star \lambda = (\xi\star \lambda)\pcirc p$ for any $\xi\in Q\otimes_A U$, $p\in U$, $a\in A$ and $\lambda\in L$, and so $(\xi \pcirc p)\star p'= (\xi \star p')\pcirc p$ for any $\xi\in Q\otimes_A U$, $p,p'\in U$.

Let us also notice that both structures induce by scalar restriction two different $A$-module structures on $Q\otimes_A U$: for the first one $(q\otimes p)\pcirc a = q\otimes (pa)$, and for the second one $(q\otimes p)\star a = (qa)\otimes p = q\otimes (ap)$.

\begin{corollary}  \label{coro1:main-1}
Under the above conditions, there is a unique $U$-linear map $\tau: (Q\otimes_A U, \pcirc) \to (Q\otimes_A U,\star)$ such that $\tau(q\otimes 1) = q\otimes 1$ for all $q\in Q$. 
Moreover, $\tau$ is an isomorphism, $\tau(\xi \star p) = \tau(\xi)\pcirc p$ for all $\xi\in Q\otimes_A U$ and for all $p\in U$, and $\tau$ is involutive.
\end{corollary}

\begin{proof}  Let us consider the $A$-linear map $\sigma: Q \to  Q\otimes_A U$ defined by $\sigma (q) = q\otimes 1$, where we consider on $Q\otimes_A U$ the $A$-module structure induced by the second right $U$-module structure. The map $\sigma$ induces a map of right $U$-modules
 $$\tau: (Q\otimes_A U, \pcirc) \to (Q\otimes_A U, \star)$$ with $\tau(q\otimes p) = (q\otimes 1)\star p$. By applying Theorem \ref{teoapp:main-1} to $L_0=0$, $U_0=A$ and $M_0=A$ we deduce that $\tau$ is an isomorphism.
 \medskip
 
 For any $q\in Q$ and any $a\in A$ we have $\tau(q\otimes a)= \tau((q\otimes 1)\pcirc a)= (q\otimes 1)\star a = q\otimes a$, and so $\tau^2(q\otimes a)= q\otimes a$.
 \medskip
 
Let $(U_r)_{r\geq 0}$ be the canonical filtration on $U$: $ U_r$ is the left (or right) $A$-module generated by $A$ and $L^r$. Let $(\Gamma_r)_{r\geq 0}$ be the filtration on  $Q\otimes_A U$ induced by $(Q\otimes_A U_r)_{r\geq 0}$. It is clear that $\Gamma_r \pcirc U_s \subset \Gamma_{r+s}$.
\smallskip

It is not difficult to prove that $\tau(\xi \star p) = \tau(\xi)\pcirc p$ by induction on $\deg \xi + \deg p$.
\inhibe
Let us prove that $\tau(\xi \star p) = \tau(\xi)\pcirc p$ by induction on $\deg \xi + \deg p$. If $\deg \xi + \deg p=0$ then $\deg \xi = \deg p =0$ and $\xi = q\otimes 1$, $p=a\in A$: $\tau(\xi \star p) = \tau((q\otimes 1)\star a)=\tau (q\otimes a)= (q\otimes 1)\star a = q\otimes a = (q\otimes 1)\pcirc a = \tau(\xi)\pcirc a$.
 \smallskip
 
Let $r>0$ be an integer and assume that $\tau(\xi' \star p') = \tau(\xi')\pcirc p'$ whenever that $\deg \xi' + \deg p' < r$. Let us take elements $\xi\in Q\otimes_A U$, $p\in U$ with $\deg \xi + \deg p=r$. If $\deg \xi =0$, then $\xi = q\otimes 1$. We can suppose that $p=\lambda p' $ with $\deg p' = r-1$ and $\lambda\in L$. We have
\begin{eqnarray*} &\xi \star p = (\xi \star \lambda)\star p'= ((q\lambda)\otimes 1 - q\otimes \lambda)\star p' =&\\
&((q\lambda)\otimes 1)\star p' - ((q\otimes 1)\pcirc \lambda)\star p' =
((q\lambda)\otimes 1)\star p'  -  ((q\otimes 1)\star p')\pcirc \lambda,
\end{eqnarray*}
and so 
\begin{eqnarray*} &\tau(\xi \star p) =\cdots = \tau(((q\lambda)\otimes 1)\star p')  -  \tau(((q\otimes 1)\star p')\pcirc \lambda)=& \\
& \tau((q\lambda)\otimes 1)\pcirc p' - \tau((q\otimes 1)\star p')\star \lambda=\tau((q\lambda)\otimes 1)\pcirc p' - (\tau(q\otimes 1)\pcirc p')\star \lambda=&\\
&  ((q\lambda)\otimes 1)\pcirc p' - ((q\otimes 1)\pcirc p')\star \lambda = (q\lambda)\otimes p' - (q\otimes p')\star \lambda=&\\
& (q\lambda)\otimes p'  - [ (q\lambda)\otimes p'- q\otimes (\lambda p')] = q\otimes p = \tau(\xi) \pcirc p.&
\end{eqnarray*}
If $\deg \xi > 0$ we can suppose $\xi = \xi' \pcirc \lambda$ with $\deg \xi' = \deg \xi -1$, $\lambda\in L$:
\begin{eqnarray*} & \tau (\xi \star p) = \tau((\xi' \pcirc \lambda)\star p) = \tau((\xi'\star p)\pcirc \lambda) = \tau(\xi'\star p) \star \lambda = (\tau(\xi')\pcirc p)\star \lambda=&\\
& (\tau(\xi')\star \lambda)\pcirc p = \tau(\xi'\pcirc \lambda)\pcirc p = \tau(\xi) \pcirc p.&
\end{eqnarray*}
\endinhibe
The involutivity of $\tau$ comes from the fact that $\tau^2(q\otimes p) = \tau((q\otimes 1)\star q) = (q\otimes 1)\pcirc p = q\otimes p$.
\end{proof}

\begin{corollary}  \label{coro2:main-1}
If $M$ is a free left $U$-module and $Q$ is a right $U$-module which is free over $A$, then $Q\otimes_A M$ is a free right $U$-module.
\end{corollary}

\begin{proof} Since $M$ is free, $M \simeq \oplus_{i\in I} U$ for some set $I$, and $Q\otimes_A M \simeq \oplus_{i\in I} (Q\otimes_A U,\star)$, but $(Q\otimes_A U,\star) \simeq (Q\otimes_A U,\pcirc)$ is a free right $U$-module because $Q$ is a free $A$-module.
\end{proof}

\begin{corollary}  \label{coro3:main-1}
Assume that $Q$ is a right $U$-module which is free over $A$. Then, for any upper bounded complex $M_0^\sbullet$ of left $U_0$-modules there is a canonical isomorphism 
$$\tau: [Q\otimes_A M^\sbullet_0]\Lotimes_{U_0} U \to
Q\otimes_A[U\Lotimes_{U_0} M^\sbullet_0]$$
in the upper bounded derived category of right $U$-modules.
\end{corollary}

\begin{proof} Let $L^\sbullet \to M_0^\sbullet$ be a free resolution. Since $Q$ is a free $A$-module, we have an isomorphism $Q \otimes_A L^\sbullet \xrightarrow{\simeq} Q\otimes_A M^\sbullet_0$ in the upper bounded derived category of right $U_0$-modules, and we can apply Corollary \ref{coro2:main-1} to deduce that $Q \otimes_A L^\sbullet$ is an upper bounded complex of free right $U_0$-modules. So we have 
$$ [Q\otimes_A M^\sbullet_0]\Lotimes_{U_0} U \simeq [Q\otimes_A L^\sbullet]\Lotimes_{U_0} U \simeq [Q\otimes_A L^\sbullet]\otimes_{U_0} U$$
and 
$ Q\otimes_A[U\Lotimes_{U_0} M_0] \simeq Q\otimes_A[U\otimes_{U_0} L^\sbullet]$.
To conclude, we apply Theorem \ref{teoapp:main-1}.
\end{proof}
\bigskip

Let $M$ be a left $U$-module and $N_0$ an $A$-module. The module $U \otimes_A N_0$ is, by scalar extension, a left $U$-module:
$ p(q \otimes n) = (pq) \otimes n$, $p,q\in U, n\in N$. Let us denote by $[U \otimes_A N_0]  \otimes'_A M$ the tensor product with respect to the $A$-module structure on $U \otimes_A N_0$ induced by its left $U$-module structure:
$$ a ( [q\otimes n] \otimes' m) = ( a[q\otimes n] ) \otimes' m = [(aq)\otimes n] \otimes' m =
 [q\otimes n] \otimes' (am)$$
for $a\in A, q\in U, n\in N_0, m\in M$. It has a left $U$-module structure given by \ref{internal-oper}:
$$ \lambda ( [q\otimes n] \otimes' m) = [(\lambda q)\otimes n] \otimes' m + [q\otimes n] \otimes' (\lambda m),\quad \lambda \in L, q\in U, n\in N_0, m\in M.$$
The map 
$$\sigma':
N_0\otimes_A M \to [U \otimes_A N_0]  \otimes'_A M,\quad \sigma'(n\otimes
m) =   [1\otimes n] \otimes' m$$
is $A$-linear and induces a $U$-linear map
$\tau': U\otimes_A [N_0\otimes_A M] \to
 [U \otimes_A N_0] \otimes'_A  M
$ with $\tau'(p \otimes [n\otimes m]) = p\cdot \sigma'(n\otimes m)$.

\begin{proposition} \label{prop:nuevomain-2} Under the above hypotheses, the map $$\tau':U\otimes_A [N_0\otimes_A M] \to
 [U \otimes_A N_0] \otimes'_A M$$ is an isomorphism of left $U$-modules. 
\end{proposition}

\begin{proof} It can be done in a completely similar way to the proof of theorem \ref{teoapp:main-1}, but we use lemma \ref{lema:tecnico-1} instead of lemma \ref{lema:tecnico-2}.
\end{proof}

In the special case where $N_0=A$, we obtain a reacher statement. The left $U$-modules $U\otimes_A M$ and $U\otimes'_A M$ are endowed with compatible right $U$-module structures \ref{internal-oper}:
\begin{eqnarray*}
&(q\otimes m) a  = q\otimes (am) = (qa)\otimes m, \quad  (q\otimes m)\lambda = (q\lambda)\otimes m - q\otimes (\lambda m),&\\
& (q\otimes' m) p = (qp)\otimes' m\quad \text{(remember that $(aq)\otimes' m = q\otimes' (am)$)}&
\end{eqnarray*}
for $a\in A, \lambda \in L, p, q\in U, m\in M$.

\begin{proposition}  \label{prop:nuevomain-3}
The left $U$-linear map $\tau': U\otimes_A M \to U \otimes'_A M$ defined by $\tau'(p\otimes m) = p\cdot (1\otimes' m)$  is also right $U$-linear and so it is an isomorphism of $(U;U)$-bimodules.
\end{proposition}

\begin{proof} The fact that $\tau'$ is an isomorphism of left $U$-modules comes from Proposition \ref{prop:nuevomain-2}. The fact that $\tau'$ is right $U$-linear can be proven in a similar way to Corollary \ref{coro1:main-1}. 
\end{proof}   

\begin{remark} \label{variants}Proposition \ref{prop:nuevomain-2} is a particular case of a second associativity law $U\otimes_{U_0} [N_0\otimes_A M] \simeq [U \otimes_{U_0} N_0] \otimes'_A M $ with $L_0=0$ and $U_0=A$. Actually, there is a third associativity law $\left[Q_0 \otimes_A M\right] \otimes_{U_0} U \simeq \left[Q_0 \otimes_{U_0} U\right] \otimes_A M$. Both laws and their corresponding corollaries can be proved in a similar way to Theorem \ref{teoapp:main-1}. For that one needs some variants of Lemmas \ref{lema:varios-isos} -- \ref{lema:tecnico-2}. We do not need these results in this paper and we skip their proof.
\end{remark}
\bigskip

From now on, we assume that $L$ is a Lie-Rinehart algebra over $(k,A)$ which is a free $A$-module of rank $d$. 
By the Poincar\'e-Birkhoff-Witt theorem \cite[th. 3.1]{rine-63} we know that $\Sim L \simeq \gr U$. 
\medskip

Let us denote $\Omega_L = L^* = \Hom_A(L,A)$, $\Omega_L^n = \bigwedge^n \Omega_L \equiv \Hom_A(\bigwedge^n L,A)$ for $n=0,\dots, d$, and $\omega_L = \Omega_L^d$.
\medskip

\numero \label{nume:SP} For each left $U$-module $E$ we define the {\em Cartan-Eilenberg-Chevalley-Rinehart-Spencer} $\SP^\sbullet_L(E)$ complex associated with $E$ 
as \cite[1.1.2]{calde_nar_lct_ilc}:
$$\SP^{-n}_{L}(E) = U\otimes_A
\stackrel{n}{\bigwedge}L \otimes_A E,\quad n=0,\dots, d$$ where the left $U$-module structure comes exclusively from the first factor $U$ of the tensor product,  
the differential 
 $d^{-n}: \SP^{-n}_{L}(E) \to \SP^{-n+1}_{L}(E)$ is given by:
\begin{eqnarray*}
&\displaystyle d^{{-1}}(P\otimes\lambda \otimes e ) = (P\lambda)\otimes e - P\otimes (\lambda e), &\\
&\displaystyle d^{{-n}}( P\otimes(\lambda_1\wedge\cdots\wedge\lambda_n)\otimes e)
 =\sum_{i=1}^n
(-1)^{i-1} P\lambda_i\otimes(\lambda_1\wedge\cdots\widehat{\lambda_i}
\cdots\wedge\lambda_n)\otimes e&\\ 
&\displaystyle  - \sum_{i=1}^n
(-1)^{i-1} P\otimes(\lambda_1\wedge\cdots\widehat{\lambda_i}
\cdots\wedge\lambda_n)\otimes (\lambda_i e)&\\
&\displaystyle + \sum_{1\leq i<j\leq n}
(-1)^{i+j}P\otimes([\lambda_i,\lambda_j]\wedge\lambda_1\wedge\cdots
\widehat{\lambda_i}\cdots\widehat{\lambda_j}
       \cdots\wedge\lambda_n)\otimes e, \  2\leq n\leq d,&
\end{eqnarray*}
and the augmentation  is
$ P\otimes e \in U\otimes_A E = \SP^{0}_{L}(E)\mapsto d^0(P\otimes e):= P\cdot e\in E$.
\medskip

In the case $E=A$ the above complex coincides with the complex defined in \cite[\S 4]{rine-63}, which will be simply denoted by $\SP^\sbullet_L$. 
\medskip

Let us denote by $\widetilde{\SP^\sbullet_L}$ and $\widetilde{\SP^\sbullet_L}(E)$ the augmented complexes $\SP^\sbullet_L \to A$ and $\SP^\sbullet_L(E) \to E$ respectively.
\medskip

We quote the following result (cf. \cite[Proposition 1.7]{calde_nar_lct_ilc}):

\begin{proposition} \label{prop:SpE} Under the above hypotheses, if $E$ is free over $A$, then $\SP^\sbullet_L(E)$ is a free resolution of the $U$-module $E$.
\end{proposition}

For any left $U$-module $M$, the complex $\Hom_U(\SP^\sbullet_L,M)$ is canonically isomorphic to the ``de Rham'' complex $\Omega_L^\sbullet(M)$, where $\Omega^n_L(M):= \Hom_A(\bigwedge^n L,M) \equiv \Hom_A(\bigwedge^n L,A) \otimes_A M =\Omega_L^n \otimes_A M$, $n=0,\dots, d$, and the differential $d:\Omega^{n-1}_L(M) \to \Omega^n_L(M)$ is given by
\begin{eqnarray*}
&\displaystyle  (d \alpha) (\lambda_1\wedge\cdots\wedge\lambda_n))
 =\sum_{i=1}^n
(-1)^{i-1} \lambda_i \alpha(\lambda_1\wedge\cdots\widehat{\lambda_i}
\cdots\wedge\lambda_n)&\\ 
&\displaystyle + \sum_{1\leq i<j\leq n}
(-1)^{i+j} \alpha([\lambda_i,\lambda_j]\wedge\lambda_1\wedge\cdots
\widehat{\lambda_i}\cdots\widehat{\lambda_j}
       \cdots\wedge\lambda_n).&
\end{eqnarray*}

For any $\lambda\in L$ and any $n$ we have the Lie derivative $L_\lambda: \Omega^n_L(M) \to \Omega^n_L(M)$ 
(cf. \cite{rine-63}, prop. 6.3) defined by 
$$ (L_\lambda \alpha)(\lambda_1\wedge\cdots\wedge\lambda_n) = \lambda\cdot \alpha(\lambda_1\wedge\cdots\wedge\lambda_n) - \sum_{i=1}^n \alpha( \lambda_1\wedge\cdots\wedge [\lambda,\lambda_i]\wedge \cdots \wedge\lambda_n).$$

The proof of the following proposition can be found in \cite{MR2000f:53109}, prop. 2.8 and th. 2.10 (see also \cite[Proposition 3.1]{calde_nar_fourier}). It is a generalization of the well known corresponding results in $D$-module theory.

\begin{proposition} \label{prop:omega}  Under the above hypotheses, the following properties hold:
\begin{enumerate}
\item[(1)] The action $(\alpha,\lambda)\in \omega_L \times L \mapsto -L_\lambda \alpha \in \omega_L$ defines a right $U$-module structure on $\omega_L$.
\item[(2)] The complex $\Omega^\sbullet_L(U)$ can be augmented through the map $\alpha \otimes p \in \Omega_L^d \otimes_A U \mapsto \alpha p \in \omega_L=\Omega_L^d$ and it becomes a free resolution of the right $U$-module $\omega_L$.
\end{enumerate}
\end{proposition}

\begin{corollary} \label{coro:omega}  The complex of right $U$-modules $\RR \Hom_U(A,U)$ is concentrated in degree $d$ and its $d$-cohomology $\Ext_U^d(A,U)$ is canonically isomorphic to the right $U$-module $\omega_L$.
\end{corollary}

\begin{proof} It is a consequence of the above propositions and the canonical isomorphisms $\RR \Hom_U(A,U) \simeq \Hom_U(\SP^\sbullet_L,U) \simeq \Omega^\sbullet_L(U)$.
\end{proof}

\begin{definition} \label{def:dualizing}
The right $L$-module $\omega_L$ is called the {\em dualizing module} of the Lie-Rinehart algebra $L$.
\end{definition}
\medskip

\begin{proposition} \label{prop:SpE_comp}
With the above notations, we have a canonical isomorphism of left $U$-modules
$ \widetilde{\SP^\sbullet_L}(E) \simeq  \widetilde{\SP^\sbullet_L} \otimes'_A E  $
where the tensor product $\otimes'_A$ is taken with respect to the $A$-module structure on each $\SP^{-n}_{L}$ induced by its left $U$-module structure.
\end{proposition}

\begin{proof} From Proposition \ref{prop:nuevomain-2} there are canonical isomorphisms
$$ \alpha^{-n}: \SP^{-n}_L(E) = U \otimes_A \left(\stackrel{n}{\bigwedge}L \otimes_A E\right) \simeq \left( U\otimes_A \stackrel{n}{\bigwedge}L\right) \otimes'_A E =
\SP^{-n}_L \otimes'_A E 
$$
for $n=0,\dots,d$. To prove the commutativity of the diagrams
\begin{equation*}
\begin{CD}
U \otimes_A \left(\stackrel{n}{\bigwedge}L \otimes_A E\right)  @>{\alpha^{-n}}>>  \left( U\otimes_A \stackrel{n}{\bigwedge}L\right) \otimes'_A E\\
@V{d^{-n}}VV    @VV{d^{-n}\otimes' \Id}V  \\
U \otimes_A \left(\stackrel{n-1}{\bigwedge}L \otimes_A E\right)  @>{\alpha^{-n+1}}>>  \left( U\otimes_A \stackrel{n-1}{\bigwedge}L\right) \otimes'_A E
\end{CD}
\end{equation*}
it is enough to check that the maps $\left( d^{-n}\otimes' \Id \right) \pcirc \alpha^{-n}$ and $\alpha^{-n+1} \pcirc d^{-n}$ coincide on elements of the form $1\otimes((\lambda_1\wedge\cdots\wedge\lambda_n)\otimes e)$, and this a straightforward computation. The commutativity corresponding to the augmentations 
\begin{equation*}
\begin{CD}
U \otimes_A  E  @>{\alpha^{0}}>>   U \otimes'_A E\\
@V{d^{0}}VV    @VV{d^{0}\otimes' \Id}V  \\
 E  @>{\text{\ref{internal-oper-nat-iso}}}>>  A \otimes'_A E
\end{CD}
\end{equation*}
can be checked in a similar way.

\end{proof}

\begin{corollary} \label{coro:dual-E}  Let $E$ be a left $U$-module which is a free $A$-module of finite rank. 
The complex of right $U$-modules $\RR \Hom_U(E,U)$ is concentrated in degree $d$ and its $d$-cohomology $\Ext_U^d(E,U)$ is canonically isomorphic to the right $U$-module $\omega_L \otimes_A E^*$.
\end{corollary}

\begin{proof} We have canonical $U$-right isomorphisms
\begin{eqnarray*}
&\RR \Hom_U(E,U) \stackrel{\ref{prop:SpE}}{\simeq} \Hom_U(\SP^\sbullet_L(E),U) \stackrel{\ref{prop:SpE_comp}}{\simeq} \Hom_U(\SP^\sbullet_L \otimes'_A E,U)  \stackrel{\ref{lema:tecnico-1}}{\simeq} &\\
& \Hom_U(\SP^\sbullet_L, \Hom_A(E,U)) \stackrel{\ref{coro:varios-can-hom}}{\simeq}  \Hom_U(\SP^\sbullet_L, E^*\otimes_A U) \stackrel{\ref{internal-oper-nat-iso}}{\simeq} &\\
& \Hom_U(\SP^\sbullet_L, U\otimes'_A E^*)   \stackrel{\ref{prop:nuevomain-3}}{\simeq} \Hom_U(\SP^\sbullet_L, U \otimes_A E^* )
\simeq &\\
& \Hom_U(\SP^\sbullet_L, U) \otimes_A E^* \simeq 
  R \Hom_U(A,U) \otimes_A E^* \stackrel{\ref{coro:omega}}{\simeq} \omega_L[-d] \otimes_A E^*.&
\end{eqnarray*}
\end{proof}

For any left $U$-module $M$, the $A$-module $\omega_L\otimes_A M$  has a canonical right $U$-module structure \ref{internal-oper} and it will be denoted by $M^{\text{right}}$. Similarly, for any right $U$-module $Q$, the $A$-module $\Hom_A(\omega_L,Q)$  has a canonical left $U$-module structure \ref{internal-oper} and it will be denoted by $Q^{\text{left}}$. Since $\omega_L$ is a free $A$-module of rank 1, the canonical maps 
$$  M \to ( M^{\text{right}} )^{\text{left}},\quad ( Q^{\text{left}} )^{\text{right}}\to Q$$
are $U$-linear isomorphisms (see Lemmas \ref{lema:varios-can-hom-2} and \ref{lema:varios-eval-hom}).
\medskip

Functors $M \mapsto M^{\text{right}}$ and $Q \mapsto Q^{\text{left}}$ form a couple of quasi-inverse exact functors and they establish an equivalence between the category of left $U$-modules and the category of right $U$-modules.
\medskip

\begin{proposition} \label{prop:finite-ghd} Assume that $A$ is an equicodimensional Noetherian regular ring. Then, the enveloping algebra $U$ of $L$ is a left and right Noetherian ring of finite global dimension $\leq \dim A + d$.
\end{proposition}

\begin{proof} By the Poincar\'e-Birkhof-Witt theorem, we know that $\gr U \simeq \Sim L$ is a commutative  Noetherian regular ring of pure graded dimension equal to $\dim A + d$. We deduce that $U$ is a ring of ``Diff'' type of global homological dimension $\leq \dim \gr U$  (cf. \cite[1.1]{meb_nar_dmw}).
\end{proof}

From now on we assume that $A$ is an equicodimensional Noetherian regular ring. 
\medskip

Let $\cd{b}{f}{U}$  be  the bounded derived category of left $U$-modules with finitely generated cohomologies.

\begin{definition}  \label{def:dual}
The duality functor $\mathbb{D}_L: \cd{b}{f}{U} \to \cd{b}{f}{U}$ is defined by
$$
\mathbb{D}_L (M) = \left(\RR \Hom_{U}(M, U)[d]\right)^{\text{\rm left}}.
$$
It is a contravariant involutive triangulated functor.
\end{definition}
\medskip

The following proposition is a consequence of Corollary \ref{coro:dual-E}.

\begin{proposition} \label{prop:dual-connection} For each left $U$-module $E$ which is a free $A$-module of finite rank we have a canonical isomorphism $\mathbb{D}_L(E) \simeq E^*$.
\end{proposition} 

Let $L_0\to L$ be a map of Lie-Rinehart algebras over $(k,A)$ and assume that $L_0$ and $L$ are free $A$-modules of rank $d_0$ and $d$ respectively. Denote by $U_0 \to U$ the induced ring map between their enveloping algebras. 

\begin{definition} \label{def:rel-dualizing}
The {\em relative dualizing module} is by definition $\omega_{L/L_0} = \Hom_A(\omega_L,\omega_{L_0})$. It has a natural left $U_0$-module structure by \ref{internal-oper}.
\end{definition}

The relative dualizing module $\omega_{L/L_0} $ is a free $A$-module of rank 1.

\begin{example} \label{ejem:app-2}
Let $h:(\CC^d,0)\to (\CC,0)$ be a reduced equation of a germ of a free divisor $D\subset (\CC^d,0)$, $k=\CC$, $A= \OO_{\CC^d,0}$ is the ring of germs at the origin of holomorphic functions, $L_0=\fDer(-\log D)_0$ is the Lie-Rinehart algebra of germs of logarithmic vector fields with respect to $D$ and $L=\Der_\CC(\OO_{\CC^d,0})$. 
The dualizing module $\omega_L$ is the module of germs of top holomorphic forms $\Omega^d_{\CC^d,0}$ and the dualizing module
$\omega_{L_0}$ is the module of germs of top logarithmic forms $\Omega^d_{\CC^d,0}(-\log D)$, i.e. $\omega_{L_0} = \Omega^d_{\CC^d,0}(D) =(1/h)\cdot \omega_L$. So, the relative dualizing module $\omega_{L/L_0}$ is the module $\OO_{\CC^d,0}(D)= (1/h)\cdot \OO_{\CC^d,0}$ of germs of meromorphic functions with poles along $D$ of order at most 1.
\end{example}
\medskip

The following lemma is a consequence of Lemma \ref{lema:varios-eval-hom}.

\begin{lemma} \label{lema:aux} The canonical map $\omega_L \otimes_ A \omega_{L/L_0}  \to \omega_{L_0}$ is an isomorphism of right $U_0$-modules.
\end{lemma}

\begin{remark} When we take $L=\Der_k(A)$ and the map $L_0 \to L$ as the anchor map of $L_0$, the relative dualizing module $\omega_{\Der_k(A)/L_0}$ coincides with the $A$-dual of the module $\mathcal{Q}_L$ considered in \cite[(6.6)]{MR2000f:53109}.
\end{remark}

The following theorem is a slight modification of Theorem 4.5 in \cite{calde_nar_ferrara} (see also \cite[Corollary 3.1.2]{calde_nar_fourier}). The existence of a such result has been suggested by \cite[Appendix A, Proposition (A.2)]{es_vi_86}, \cite[Proposition 2.2.5]{ucha_tesis} and \cite[Theorem 4.3]{cas_ucha_stek}. A related result is \cite[Lemma 4.3.3]{MR2000m:32015}.

\begin{theorem} \label{teo:duality-ferrara} Under the above hypotheses, for every bounded complex $M$ of left $U_0$-modules there is a canonical isomorphism in $\cd{b}{f}{U}$
$$ \mathbb{D}_L \left( U \Lotimes_{U_0} M \right) \simeq U \Lotimes_{U_0} \left(\omega_{L/L_0}[d-d_0] \stackrel{\phantom{L}}{\otimes}_A \mathbb{D}_{L_0}(M)\right).
$$
\end{theorem} 

\begin{proof}  From Corollary \ref{coro3:main-1} and Lemma \ref{lema:aux}, there are canonical isomorphisms 
\begin{eqnarray*}
&\omega_L \otimes_A \left( U \Lotimes_{U_0} \left(\omega_{L/L_0} \stackrel{\phantom{L}}{\otimes}_A \mathbb{D}_{L_0}(M)\right) \right) \simeq &\\
&\left( \omega_L \stackrel{\phantom{L}}{\otimes}_A \left(\omega_{L/L_0} \stackrel{\phantom{L}}{\otimes}_A \mathbb{D}_{L_0}(M)\right) \right)   \Lotimes_{U_0} U\simeq &\\
&  \left( \left(\omega_L \otimes _A \omega_{L/L_0} \right)  \stackrel{\phantom{L}}{\otimes}_A \mathbb{D}_{L_0}(M)\right)   \Lotimes_{U_0} U \simeq 
\left( \omega_{L_0}  \otimes_A \mathbb{D}_{L_0}(M)\right)  \Lotimes_{U_0} U.
\end{eqnarray*}
By definition of $\mathbb{D}$, we have canonical isomorphisms
\begin{eqnarray*}
& \omega_L \otimes_A \mathbb{D}_L \left( U \Lotimes_{U_0} M \right) \simeq  \RR \Hom_{U}\left(U \Lotimes_{U_0} M, U\right)[d] \simeq &\\
& \RR \Hom_{U_0}(M, U)[d] \simeq 
 \left(\RR \Hom_{U_0}(M, U_0) \Lotimes_{U_0} U\right) [d] \simeq &\\ 
 &\left(\left(\omega_{L_0} \otimes_A \mathbb{D}_{L_0}(M)\right)[-d_0] \Lotimes_{U_0} U\right) [d] \simeq &\\
& \left(\left(\omega_{L_0} \otimes_A \mathbb{D}_{L_0}(M)\right) \Lotimes_{U_0} U\right) [d-d_0].
\end{eqnarray*}
We conclude by applying the functor $\left( - \right)^{\text{left}}$.
\end{proof}

\bigskip

{\small \noindent Departamento de \'Algebra \& Instituto de Matem\'aticas (IMUS)\\
 Facultad de  Matem\'aticas, Universidad de Sevilla,\\ Apdo. 1160, 41080
 Sevilla, Spain}. \\
{\small {\it E-mail}:  narvaez@algebra.us.es
 }

\end{document}